\documentclass[12pt,leqno]{amsart}
\usepackage{amsmath,amssymb,amsthm,graphicx,mathrsfs,url}
\usepackage[usenames,dvipsnames]{color}
\usepackage[colorlinks=true,linkcolor=Red,citecolor=Green]{hyperref}
\usepackage{amsxtra,esint}
\usepackage{mathtools}
\usepackage{hyperref}
\usepackage{MnSymbol}

\usepackage{dsfont}
\usepackage{commath}
\usepackage{comment}
\usepackage{cleveref}
\usepackage{enumerate}

\usepackage{float}
\usepackage[font=small,labelfont=bf]{caption}

\def\?[#1]{\textbf{[#1]}\marginpar{\Large{\textbf{??}}}}

\theoremstyle{plain}
\newtheorem{prop}{Proposition}[section]
\newtheorem{defi}[prop]{Definition}
\newtheorem*{theorem*}{Theorem}
\newtheorem{theorem}[prop]{Theorem}

\newtheorem{lemma}[prop]{Lemma}
\newtheorem{corr}[prop]{Corollary}

\theoremstyle{definition}
\newtheorem{rem}[prop]{Remark}

\numberwithin{equation}{section}

\newtheorem{question}[prop]{Question}

\newcommand{\N}{\mathbb{N}}

\newcommand{\Z}{\mathbb{Z}}

\newcommand{\R}{\mathbb{R}}

\newcommand{\C}{\mathbb{C}}
\renewcommand{\phi}{\varphi}

\newcommand{\uno}{\mathds{1}}

\renewcommand{\Re}[1]{{\,\mathfrak{Re}} \left ( #1\right ) }

\newcommand{\eps}{\varepsilon}

\DeclareMathOperator{\spn}{span}

\begin{document}

\title[On the existence of large SPR subspaces of $C(K)$]{On the existence of large subspaces of $C(K)$ that perform stable phase retrieval}

\author{Enrique Garc\'ia-S\'anchez}
\address{Instituto de Ciencias Matem\'aticas (CSIC-UAM-UC3M-UCM)\\
Consejo Superior de Investigaciones Cient\'ificas\\
C/ Nicol\'as Cabrera, 13--15, Campus de Cantoblanco UAM\\
28049 Madrid, Spain.
\newline
	\href{https://orcid.org/0009-0000-0701-3363}{ORCID: \texttt{0009-0000-0701-3363} } }
\email{enrique.garcia@icmat.es}

\author{David de Hevia}
\address{Instituto de Ciencias Matem\'aticas (CSIC-UAM-UC3M-UCM)\\
Consejo Superior de Investigaciones Cient\'ificas\\
C/ Nicol\'as Cabrera, 13--15, Campus de Cantoblanco UAM\\
28049 Madrid, Spain.
\newline
\href{https://orcid.org/0009-0003-5545-0789}{ORCID: \texttt{0009-0003-5545-0789}}}
\email{david.dehevia@icmat.es // davhevia@ucm.es}
\author{Mitchell~A.\ Taylor}
\address{Department of Mathematics\\
ETH Z\"urich, Ramistrasse 101, 8092 Z\"urich, Switzerland.
} \email{mitchell.taylor@math.ethz.ch}



\begin{abstract}
The purpose of this article is to address an open problem posed by Freeman-Oikhberg-Pineau-T.~(\textit{Math.~Ann.}~2024) regarding the existence of large subspaces of $C(K)$ that perform stable phase retrieval (SPR). We begin by proving that for both the real and complex fields, the space $C(K)$ admits an infinite-dimensional SPR subspace if and only if the second Cantor-Bendixson derivative $K{''}$ is nonempty. We then show how to construct ``large" SPR subspaces of $C(K)$, where the size of the subspace depends quantitatively on the number of non-trivial Cantor-Bendixson derivatives that the compact Hausdorff space $K$ possesses.
\end{abstract}

\maketitle

\section{Introduction}

In many areas of physics and engineering, one is tasked with recovering a signal $f\in H$ from a phaseless measurement $|Tf|\in X$ where $T:H\to X$ is a linear operator mapping a signal space $H$ to a function space $X$. Such problems appear prominently in crystallography, as detectors can only record the modulus of the Fourier transform of the desired signal. This leads to the mathematical task of recovering the true signal $f$ from the distorted image $|\mathcal{F}f|$, under the constraint that $f$ lies in a known subspace $H$ of $L_2(\mathbb{R}^d)$ consisting of functions relevant to the experiment. Note that for any unimodular scalar $\lambda$, the signals $f, \lambda f\in H$ give rise to the same phaseless Fourier measurements. Thus, the objective of \emph{Fourier phase retrieval} is to stably recover $f\in H$ from $|\mathcal{F}f|\in L_2(\mathbb{R}^d)$ up to global phase ambiguity.\medskip

Although Fourier phase retrieval is the prototypical example of a phase retrieval problem, many other interesting variants arise from different physical models. Given $T:H\to X$ as above, we say that $T$ does \emph{phase retrieval} if for all $f,g\in H$ we have $|Tf|=|Tg|$ if and only if $f=\lambda g$ for some unimodular scalar $\lambda$. 
Fourier phase retrieval is thus the special case where $T=\mathcal{F}|_H:H\subseteq L_2(\mathbb{R}^d)\to L_2(\mathbb{R}^d).$ Below, we list several other important situations where phase information cannot be directly measured and needs to be mathematically recovered. For a more comprehensive discussion, we refer the reader to the surveys \cite{surveyGKR,surveyJEH}.
\begin{enumerate}
 \item \textbf{(Ptychography).} In modern imaging applications, it is common to insert a window $g$ in front of the crystal, so that the detector records $|\mathcal{F}(fg)|$. Translating the window $g$ and then repeating the experiment, one gains access to the phaseless short-time Fourier transform (STFT) $|V_gf|$. Thus, in this case, the measurement operator $T: L_2(\mathbb{R}^d)\to L_2(\mathbb{R}^{2d})$ is the STFT with window $g$.
    \item \textbf{(Pauli problem).} In quantum mechanics, the famous \emph{Pauli problem} asks whether one can recover a wavefunction $\psi\in L_2(\mathbb{R}^d)$ with $\|\psi\|=1$ from $|\psi|^2$ and $|\mathcal{F}\psi|^2$. Here, we note that $|\psi|^2$ and $|\mathcal{F}\psi|^2$ represent the probability densities of position and momentum, respectively. In the above language, the Pauli problem asks whether the operator $T: H\subseteq L_2(\mathbb{R}^d)\to L_2(\mathbb{R}^d)\times L_2(\mathbb{R}^d)$, $T\psi=(\psi,\mathcal{F}\psi)$, does phase retrieval. In general, it is known that when $H=L_2(\mathbb{R}^d)$, this is not the case. Nevertheless, there are ``large" subspaces of $L_2(\mathbb{R}^d)$ where Pauli phase retrieval can be performed in a stable manner \cite{CPT,FOPT}.
    \item \label{SSPR} \textbf{(Schr\"odinger flows).} Pauli's problem asks whether it is possible to recover a wave function from its position and momentum densities at a single time (or equivalently, from the position density of both its initial state and its rescaled asymptotic profile). A more reasonable but equally relevant question is whether the data $\{|\psi(t)|: t\in I\}$ determines $\psi$ uniquely, where $I$ is a set of time measurements and $\psi$ is a solution to the Schr\"odinger equation. This leads to an interesting family of phase retrieval problems, modeled by the operators $T_I:L_2(\mathbb{R}^d)\to C(I; L_2(\mathbb{R}^d))$, $T_I\psi_0=\psi|_I$.
\end{enumerate}
In physical applications, it is imperative that phase retrieval not only be possible but also be stable. To quantify this, we say that a linear operator $T:H\to X$ does \emph{stable phase retrieval} if there exists a constant $C\geq 1$ such that for all $f,g\in H$  we have 
\begin{equation}\label{SPR1}
    \inf_{|\lambda|=1}\|f-\lambda g\|\leq C \||Tf|-|Tg|\|. 
\end{equation}
The primary objective of this article is to develop techniques to address the stability of the phase retrieval problem. However, before we do this, it will be convenient to reformulate the problem in a way that decouples the operator $T$ from the nonlinear inequality.
\subsection{A reformulation of the phase retrieval problem} 
Although the above examples model very different physical phenomena, they can all be unified into a single mathematical problem. Indeed, let us say that a subspace $E$ of a function space $X$ does \emph{phase retrieval} if for all $f,g\in E$ we have $|f|=|g|$ if and only if $f=\lambda g$ for some unimodular scalar $\lambda$. Since any operator $T: H\to X$ which does phase retrieval must be injective, it is easy to see that $T$ does phase retrieval if and only if the subspace $E:=T(H)\subseteq X$ does phase retrieval. In other words, the operator $T$ determines the subspace of $X$ that we wish to do phase retrieval but otherwise does not play a role in the nonlinear recovery problem.
\medskip

When considering the stability of any nonlinear inverse problem, it is natural to assume that the forward problem is well-posed, which in our case means that the mapping $f\mapsto |Tf|$ is Lipschitz. In view of \eqref{SPR1}, it follows that any operator that performs stable phase retrieval must satisfy the relation
\begin{equation*}
    \inf_{|\lambda|=1}\|f-\lambda g\|\sim \inf_{|\lambda|=1}\|Tf-\lambda Tg\|.
\end{equation*}
Hence, again, we may relabel $f\leftrightarrow Tf$ to suppress the operator $T$ in the inequality \eqref{SPR1}. More formally, we say that a subspace $E$ of a function space $X$ does \emph{stable phase retrieval} if there exists a constant $C\geq 1$ such that for all $f,g\in E$ we have 
\begin{equation}\label{SPR intro}
     \inf_{|\lambda|=1}\|f-\lambda g\|\leq C \||f|-|g|\|. 
\end{equation}
With the above discussion in mind, it is easy to see that an operator $T:H\to X$ does stable phase retrieval if and only if the subspace $E:=T(H)$ does stable phase retrieval as a subspace of $X$. As we shall soon see, this reformulation of the phase retrieval problem will clarify many aspects of the theory. In particular, it will soon become evident that all subspaces satisfying \eqref{SPR intro} share interesting structural features that do not depend on the specific physical models that they arise from.\medskip

In the vast majority of phase retrieval papers, the function space $X$ is assumed to be $L_2(\mu)$. However, to study the phase retrieval problem for Schr\"odinger flows and other time-dependent PDEs, one must consider more general function spaces. A unified perspective on phase retrieval in function spaces was pioneered in the articles \cite{CDFF,CPT,FOPT}, which showed that a remarkable structural theory for phase retrieval problems could be developed in $L_p$-spaces or even general Banach lattices (i.e., spaces equipped with a compatible set of linear, norm, and modulus operations that make the forward problem for phase retrieval well-posed). In particular, this study revealed clear connections with many diverse areas of mathematics, including the $\Lambda(p)$-set theory of sparse Fourier series and the local and isometric theories of Banach spaces.\medskip

In this article, we will focus on the stability of the phase retrieval problem when $X$ is a space $C(K)$ of continuous function on a compact Hausdorff space $K$. In particular, we will answer an open problem from \cite{FOPT} and develop important mathematical tools that will allow us to establish stability results outside of the realm of Hilbert spaces. Notably, our results will have direct consequences for phase retrieval in vector-valued function spaces such as $C(I; L_2(\mathbb{R}^d))$, which appear in the Schr\"odinger phase retrieval problem \eqref{SSPR}.

\subsection{Summary of the results and an outline of the proof} We now proceed to define the concepts studied in this paper and to motivate our main results. We begin by formally defining stable phase retrieval.
\begin{defi}\label{Defn 1}
    Let $X$ be a Banach lattice and fix $C\geq 1$. A subspace $E$ of $X$ does \emph{$C$-stable phase retrieval} (\emph{$C$-SPR}, for short) if for all $f,g\in E$ we have the inequality 
    \begin{equation}\label{SPR Def formal}
             \inf_{|\lambda|=1}\|f-\lambda g\|\leq C \||f|-|g|\|. 
    \end{equation}
\end{defi}
We remark that the choice of scalar field $\mathbb{F}\in \{\mathbb{R},\mathbb{C}\}$ is extremely important for the phase retrieval problem. If the scalar field is $\mathbb{R}$, then there are only two phases $\{\pm 1\}$, so that the left-hand side of \eqref{SPR Def formal} reduces to $\min\{ \|f-g\|, \|f+g\|\}$. However, if the scalar field is $\mathbb{C}$, the infimum in \eqref{SPR Def formal} must be taken over the entire unit circle.
\begin{rem}
    The theory of Banach lattices provides a natural functional setting for studying many important inverse problems, including phase retrieval, declipping and ReLU recovery \cite{AFRT, FH, FOPT}. For this reason, we have stated \Cref{Defn 1} (as well as a handful of results below) in this language. However, since all of the new results in this article are specific to $C(K)$, we encourage any reader not familiar with the theory of Banach lattices to assume throughout the paper that $X$ is $C(K)$, $L_p(\mu)$ or $C(K;L_p(\mu))$ for some compact Hausdorff space $K$ and measure $\mu$. 
\end{rem}

\subsubsection{Summary of the results}
Our first main result is a complete characterization of compact Hausdorff spaces $K$ for which $C(K)$ admits an infinite-dimensional SPR subspace. This result was already proven in \cite{FOPT} for the case of real scalars; our main contribution is to extend the result to the \emph{much more difficult} complex case.

\begin{theorem}\label{thm:intro-existence-subspace-SPR-C(K)}
    Let $K$ be a compact Hausdorff space. The following statements are equivalent over both the real and complex fields:
    \begin{enumerate}[(i)]
        \item $K'$, the set of non-isolated points of $K$, is infinite.
        \item $c_0$ embeds isometrically into $C(K)$ as an SPR subspace.
        \item $C(K)$ contains an infinite-dimensional SPR subspace.
    \end{enumerate}
\end{theorem} 
\Cref{thm:intro-existence-subspace-SPR-C(K)} will be proved in \Cref{sect: existence results} (where it appears as \Cref{thm: SPR subspace iif K' infinite}). This theorem will follow from several results which likely have independent interest. For example, we will prove that if $K$ is perfect (meaning that $K'=K$) then \emph{every} separable Banach space embeds isometrically into $C(K)$ as an SPR subspace. Similarly to \Cref{thm:intro-existence-subspace-SPR-C(K)}, we remark that this statement was proved in the real case in \cite{FOPT}, but the complex case needs a more systematic argument. \medskip

From the results mentioned above, we see that the ability to construct ``large" SPR subspaces of $C(K)$ is intimately connected with the number of non-trivial Cantor-Bendixson derivatives that $K$ possesses. In particular, having $K''\neq \emptyset$ is necessary and sufficient to find \emph{some} infinite-dimensional SPR subspace of $C(K)$ and having $K$ admit a non-trivial perfect subset is necessary and sufficient for \emph{every} separable Banach space to embed into $C(K)$ as an SPR subspace. This led the authors of \cite{FOPT} to pose the following question, which essentially asks whether one can link the number of non-trivial Cantor-Bendixson derivatives of $K$ with the ``size" of the SPR subspaces of $C(K)$ in a quantitative manner. Here, we recall that for an ordinal number $\alpha$, $K^{(\alpha)}$ denotes the $\alpha$-th Cantor--Bendixson derived set, which is defined using transfinite recursion as $K^{(0)}=K$, $K^{(\alpha+1)}=(K^{(\alpha)})'$, and $K^{(\alpha)}=\bigcap_{\beta< \alpha} K^{(\beta)}$ whenever $\alpha$ is a limit ordinal. With this notation, the conjecture from \cite{FOPT} may be stated as follows:

\begin{question}[Question 6.4 in \cite{FOPT}]\label{main question} The above theorem shows that $K'$ is infinite iﬀ
$C(K)$ contains an SPR copy of $c_0$. If $K$ is ``large” enough (in terms of
the smallest ordinal $\alpha$ for which $K^{(\alpha)}$ is finite),  what SPR subspaces
(other than $c_0$) does $C(K)$ have? Note that $c_0$ is isomorphic to $c =
C[0,\omega]$ ($\omega$ is the first infinite ordinal). If $K^{(\alpha)}$ is infinite, does $C(K)$
contain an SPR copy of $C[0,\omega^\alpha]$?
\end{question}

\Cref{sect: Timur's question} is dedicated to the study of this question. We would like to point out that an \emph{isometric version} of this question was already examined in \cite{GH}, although only the situation $\alpha<\omega$ was considered. A priori, this is not useful for answering \Cref{main question} as $C[1,\omega^\alpha]$ and $c_0$ are isomorphic whenever $1\leq\alpha<\omega$. However, the strategy developed in \cite{GH} will be the starting point for how we will approach this new question. Specifically, it was shown in \cite[Proposition 3.5]{GH} (see \Cref{prop: implication K_alpha copy of Comegaalpha}) that if $K^{(\alpha)}\neq\varnothing$ for some $\alpha<\omega_1$, then there exists an isometric embedding $T:C[1,\omega^\alpha]\hookrightarrow C(K)$ which preserves SPR subspaces. This result will allow us to work in a much more manageable Banach space when investigating \Cref{main question}.
 \medskip

The main result of \Cref{sect: Timur's question}  is \Cref{thm: real SPR copy of omega^alpha in omega^alpha2}, which we restate as \Cref{Theorem intro} below. We remark that this theorem admits an appropriate complex version (see \Cref{thm: complex SPR copy of omega^alpha in 2omega^alpha2}) but we will focus our exposition on the case of real scalars, which is already quite difficult.
\begin{theorem}\label{Theorem intro}
    Let $K$ be a compact Hausdorff space, $\alpha\in \omega_1$ a countable ordinal, and assume that the field of scalars is $\R$. If $K^{(\alpha\odot 2)}$ is nonempty, then there exists a $3$-SPR isometric embedding of $C[1,\omega^\alpha]$ into $C(K)$.
\end{theorem}
In \Cref{Theorem intro}, $\alpha\odot 2$ denotes the ordinal obtained by doubling all of the coefficients in the Cantor normal form expansion of $\alpha$ (see \Cref{OOTP} for a precise definition). In general, this additional assumption is necessary. Indeed, it was shown in \cite{GH} that if $|K''|=1$, then $C(K)$ cannot contain an SPR copy of $C[1,\omega^2]$, but if $|K''|\geq 2$, it always does. On the other hand, if \Cref{main question} is interpreted in the \emph{isomorphic} sense (i.e., if we only require the existence of some SPR isomorphic embedding into $C(K)$ rather than the $3$-SPR isometric embedding guaranteed by \Cref{Theorem intro}), then it follows immediately from \Cref{Theorem intro} and the isomorphic classification of separable $C(K)$-spaces by Bessaga and Pe\l czy\'nski \cite{BP} that \Cref{main question} indeed has a positive answer when $\alpha$ is a countable ordinal (see Corollaries \ref{coro: real best SPR embedding into omega^alpha} and \ref{coro: complex best SPR embedding into omega^alpha}; note that the assumption of a countable ordinal is the best one can hope for by basic density character arguments).  Indeed, recall that every separable infinite-dimensional $C(K)$-space is isomorphic to $C[0,1]$ or $C[1, \omega^{\omega^\alpha}]$ for some countable ordinal $\alpha$, and all of these spaces are isomorphically distinct (see, for instance,  \cite[Section 2]{Rosenthal}).

\subsubsection{Obstructions to phase retrieval and an outline of the proof}\label{OOTP}
Our construction utilizes two recent characterizations of stable phase retrieval that clearly isolate the main instability mechanisms. The most obvious obstruction to phase retrieval is the existence of pairs of disjoint vectors in $E$, that is, non-zero functions $f,g\in E$ such that $|f|\wedge|g|=0$, or equivalently, $fg=0$. To see that such vectors give rise to an instability in phase recovery, simply note that the functions $h_+:=f+g, h_-:=f-g$ lie in $E$, are not multiples of each other, and yet satisfy $|h_+|=|h_-|$. Remarkably, when the scalar field is $\mathbb{R}$, this is the \emph{only} source of instability in the phase retrieval problem. In fact, we have the following complete characterization of SPR in the setting of $C(K)$-spaces (see also \cite[Theorem 3.4]{FOPT} and \cite[Proposition 3.4]{Eugene} for more general results).

\begin{theorem}\label{thm: real stable phase retrieval}
    Let $E$ be a subspace of a real $C(K)$-space. The following statements are equivalent:
    \begin{enumerate}
        \item\label{S1} $E$ fails SPR.
        \item\label{S2} For every $\eps>0$, there are $f,g\in S_E$ such that $\||f|\wedge|g|\|< \eps$.
        \item\label{S3} For every $\eps>0$, there are $f,g\in S_E$ such that $\|fg\|<\eps$.
    \end{enumerate}
    Moreover, $E$ does $C$-SPR if and only if all $f,g\in S_E$ satisfy $\||f|\wedge|g|\|\geq \frac{1}{C}$.
\end{theorem}
\begin{rem}
One interesting feature of \Cref{thm: real stable phase retrieval} is that it reformulates the phase retrieval problem as a problem in three distinct fields. Indeed, statement~\eqref{S1} views phase retrieval as a nonlinear inverse problem, statement~\eqref{S2} views phase retrieval as a Banach lattice problem and statement~\eqref{S3} views phase retrieval as a Banach algebra problem. Since the lattice, isometric and algebra structures of $C(K)$ all uniquely determine $K$,  this result (together with Theorems~\ref{thm:intro-existence-subspace-SPR-C(K)} and \ref{Theorem intro}) can be viewed not only as a contribution to the study of phase retrieval in function spaces, but also to the classical program initiated in \cite{Eilenberg} of characterizing ``interesting'' Banach space/lattice/algebra properties of $C(K)$ in terms of ``interesting'' topological properties of $K$ (see \cite[p.~132]{Semadeni} for a summary of such properties). 
\end{rem}

\begin{rem}
Although the equivalence of statements~\eqref{S1} and \eqref{S2} and the moreover statement in \Cref{thm: real stable phase retrieval} are valid in any Banach lattice, they are particularly useful in $C(K)$. Indeed, unlike in $L_p$-spaces where one has to quantify the relative mass on the domain where $f$ and $g$ overlap, in $C(K)$ the condition $\||f|\wedge|g|\|\geq \frac{1}{C}$ simply states that one may find a single point $t\in K$ where $|f(t)|, |g(t)|\geq \frac{1}{C}$. This will allow us to use ``combinatorial" arguments in our constructions of SPR subspaces of $C(K)$, as we now only have to construct $E$ so that for each $f,g\in S_E$ there is a point $t\in K$ where $f(t)$ and $g(t)$ both have magnitude above a given threshold.
\end{rem}

Recall that the existence of almost disjoint pairs implies the failure of complex SPR \cite[Proposition 2.5]{CGH}, but the converse is not true \cite[Example 3.7]{FOPT}. To characterize complex phase retrieval, the non-existence of disjoint pairs has to be replaced with the non-existence of linearly independent $f,g \in E$ such that $\Re{f\overline{g}}=0$ (see \cite[Proposition 2.2]{CGH}). This condition implies that for every point $t$ in the domain we have $f(t)=\pm i g(t)$, or equivalently, that $f(t)$ and $g(t)$ are perpendicular in $\C=\R^2$ (endowed with the Euclidean product of $\R^2$). This algebraic characterization of complex phase retrieval was stabilized in \cite{CGH}, resulting in the following useful reformulation of complex SPR (which we again state in $C(K)$-spaces although it is valid more generally).
\begin{theorem}\label{thm: complex stable phase retrieval}
    Let $E$ be a subspace of a complex $C(K)$-space. The following statements are equivalent:
    \begin{enumerate}
        \item $E$ fails SPR.
        \item There exists $\delta>0$ satisfying that for every $\eps>0$, there are $f,g\in S_E$ such that $\min_{\abs{\lambda}=1} \|f-\lambda g\| \geq \delta$ and $\|\Re{f\overline{g}} \|<\eps$.
    \end{enumerate}
\end{theorem}
\Cref{thm: complex stable phase retrieval} will play a central role in the proofs of both \Cref{thm:intro-existence-subspace-SPR-C(K)} and \Cref{Theorem intro}. The utility of \Cref{thm: complex stable phase retrieval} is that it transforms SPR in complex $C(K)$-spaces (which is initially formulated as a nonlinear inverse problem) into a Banach $\ast$-algebra problem. Moreover, it does so in a way that is amenable to techniques from infinite combinatorics. We remark that although we are only interested in the commutative case, non-commutative variants of this criterion have relevance in cryo-electron microscopy \cite{ABDE,BDES}, where one wishes to recover a matrix $X$ up to inherent symmetries from its operator modulus $|X|$. \medskip

We now sketch the ideas in the proofs of our two main results: \Cref{thm:intro-existence-subspace-SPR-C(K)} and \Cref{Theorem intro}. The proof of \Cref{thm:intro-existence-subspace-SPR-C(K)} in the complex case follows the same scheme of proof as in the real case but is much more intricate. Similarly to the real case, to establish that $K'' \neq \varnothing$ implies the existence of an infinite-dimensional SPR subspace, we will first show that it suffices to construct an SPR embedding $T:c_0 \to C_0[1,\omega^2)$. However, the main challenge to obtaining a complex version of this result is verifying that $T(c_0)$ performs SPR. In particular, it will no longer be sufficient to prove that every pair of norm-one functions has a uniform non-trivial overlap (\Cref{thm: real stable phase retrieval}). Instead, we will invoke the (less tractable) characterization of complex SPR from \Cref{thm: complex stable phase retrieval}. All of the technical computations needed to apply \Cref{thm: complex stable phase retrieval} will be synthesized in \Cref{lemma: suficient condition for complex SPR}, which gives a sufficient condition for complex SPR that is quite combinatorial in nature.
\medskip

The proof of \Cref{Theorem intro} follows a similar approach. First of all, we reduce the problem to the study of SPR embeddings between spaces of the form $C[1,\omega^\alpha]$ using \cite[Proposition 3.5]{GH}. Then,  we define by transfinite induction a family of overlapping maps $U_{\alpha,\beta}:C[1,\omega^\alpha]\times C[1,\omega^\beta]\rightarrow C[1,\omega^{\alpha \oplus \beta}]$ for every pair of countable ordinals $\alpha$ and $\beta$. These maps will, in turn, allow us to construct isometric embeddings from $C[1,\omega^\alpha]$ into $C[1,\omega^{\alpha\odot 2}]$ that satisfy the sufficient conditions for real and complex SPR given by \Cref{thm: real stable phase retrieval} and \Cref{lemma: suficient condition for complex SPR}, respectively.
\medskip

Throughout \Cref{sect: Timur's question} we will make extensive use of \textit{ordinal numbers}, so let us briefly recall some basic facts about these sets -- we refer the reader to standard textbooks such as \cite{Jech} for a detailed discussion of the topic.
Loosely speaking, ordinal numbers are a class of sets that can be linearly ordered by inclusion in such a way that every set of the class is the collection of all the sets that precede it; this family has a \textit{smallest} element -- the empty set -- which will be denoted by $0$. More specifically, ordinal numbers satisfy that $\alpha <\beta$ if and only if $\alpha$ is an element of $\beta$, and every ordinal can be expressed as $\alpha=\{\beta:\beta <\alpha\}$. Given an ordinal $\alpha$, we can define its \textit{successor ordinal} as $\alpha+1:=\alpha \cup\{\alpha\}$. If $\alpha\neq 0$ is \textit{not} a successor ordinal, then it is called a \textit{limit ordinal}, and it can be written as $\alpha=\bigcup_{\beta <\alpha}\beta$. Ordinal numbers admit arithmetical operations, such as addition, multiplication, and exponentiation. To be more concrete, starting at $0$, which is the empty set, we can first generate every natural number using the successor operation. Once they have been created, we may introduce the first infinite ordinal, $\omega=\N$, which is a limit ordinal. After $\omega$ comes $\omega+1,\omega+2,\ldots, \omega+n,\ldots, \omega\cdot 2,\omega\cdot 2+1,\ldots, \omega\cdot n,\ldots,\omega^2, \omega^2+1,\ldots, \omega^2+\omega, \ldots,\omega^2\cdot 2,\ldots, \omega^n,\ldots, \omega^\omega$, and the list keeps going (see \cite[Chapter 10]{ConwayGuy} for a light introduction to the topic). The first uncountable ordinal is denoted by $\omega_1$. In this paper, we will only work with countable ordinals, so we will not be using ordinals beyond $\omega_1$. Every ordinal $\alpha >0$ can be represented uniquely as
\begin{equation}\label{eq: cantor normal form}
    \alpha = w^{\beta_1}\cdot k_1+\ldots +\omega^{\beta_n}\cdot k_n,
\end{equation}
where $n\geq 1$, $\alpha \geq \beta_1>\ldots >\beta_n\geq 0$ and $k_1,\ldots, k_n$ are non-zero natural numbers. This representation is known as \textit{Cantor's normal form}. We can associate to every ordinal $\alpha$ a totally disconnected compact Hausdorff topological space called the \textit{ordinal interval} $[0,\alpha]$, which can be described as the set $[0,\alpha]=\{\beta: 0\leq \beta \leq \alpha\}$ endowed with the topology generated by the base of open sets  $\{(\beta_1,\beta_2]:0\leq \beta_1<\beta_2\leq \alpha\}$, where $(\beta_1,\beta_2]=\{\beta\in [0,\alpha]: \beta_1< \beta \leq \beta_2\}$. Note that the interval $(\beta_1,\beta_2]$ is a clopen set in this topology for every $\beta_1<\beta_2$. It should also be observed that $[1,\omega^\alpha]$ is homeomorphic to $[0,\omega^\alpha]$ whenever $\alpha\geq 1$. Thus, $C[0,\omega^\alpha]$ and $C[1,\omega^\alpha]$ are lattice isometric. It will be convenient for us to write $C[1,\omega^\alpha]$ (and we will do so throughout the text) as it allows us to obtain a slightly simpler notation when using Cantor's normal forms.
\medskip

\section{Existence of SPR subspaces in complex \texorpdfstring{$C(K)$}{}-spaces}\label{sect: existence results}

In this section, we extend two remarkable results from \cite{FOPT} concerning the existence of real SPR subspaces of $C(K)$ to the complex setting:
\begin{itemize}
    \item $C(K)$ contains an infinite-dimensional SPR subspace if and only if $K'$ is infinite if and only if $C(K)$ contains a subspace isometric to $c_0$ that performs SPR (\Cref{thm: SPR subspace iif K' infinite}).
    \item $C(K)$ contains an isometric SPR copy of every separable Banach space if and only if $K$ contains a perfect set (\Cref{thm: perfect set}).
\end{itemize}

The real counterpart of the above results can be inferred from the proof of \cite[Theorem 6.1]{FOPT}. Therefore, our approach in this section will be to dissect that proof into different parts that we will adapt appropriately to the complex case. Unless stated otherwise, throughout this section, all the spaces considered will be complex Banach spaces or lattices.

\subsection{\texorpdfstring{$C[0,1]$}{} is a universal SPR container of all separable Banach spaces} Let us recall that the well-known Banach-Mazur theorem establishes that $C[0,1]$ contains every separable Banach space isometrically. Our aim is to show that these isometric embeddings can be made to do SPR. This was already proven in \cite[Proposition 6.2]{FOPT} for real scalars. We begin by proving a separation lemma.

\begin{lemma}\label{lemma: separating functional}
    Let $E$ be a Banach space (real or complex) and $0<\delta<2$. Then for every $x\in S_E$ and $y\in B_E$ satisfying $\min_{\abs{\lambda}=1}\|x-\lambda y\|\geq \delta$ there exists $x^*\in S_{E^*}$ such that $\abs{\,\abs{x^*(x)}-\abs{x^*(y)}}\geq \frac{\delta}{3}$.
\end{lemma}

\begin{proof}
    We distinguish two cases. If $\|y\|\leq 1-\frac{\delta}{3}$, then we choose $x^*\in S_{E^*}$ to be a norming functional for $x$, so that $x^*(x)=1$ and $\abs{x^*(y)}\leq \norm{x^*}\norm{y}\leq 1-\frac{\delta}{3}$. Otherwise, if $\|y\|> 1-\frac{\delta}{3}$, then the distance of $y$ to  $\mathbb{F}x$ (the span of $x$ with $\mathbb{F}\in \{\mathbb{R},\mathbb{C}\}$) must be bounded from below. Indeed, given $\mu \in \mathbb{F}$, we have one of the following three possibilities:
    \begin{itemize}
        \item If $\abs{\mu}>1+\frac{\delta}{3}$, then we have
        \begin{equation*}
            \|y-\mu x\|\geq \|\mu x\|-\|y\|\geq \frac{\delta}{3}.
        \end{equation*}
        \item If $\abs{\mu}<1-\frac{2\delta}{3}$, then we have
        \begin{equation*}
            \|y-\mu x\|\geq \|y\|-\|\mu x\|\geq \frac{\delta}{3}.
        \end{equation*}
        \item If $1-\frac{2\delta}{3}\leq \abs{\mu}\leq 1+\frac{\delta}{3}$, then we have $\abs{1-\abs{\mu}}\leq \frac{2\delta}{3}$ and
        \begin{equation*}
            \|y-\mu x\|\geq \norm{y-\frac{\mu}{\abs{\mu}}x}-\norm{\left( \frac{\mu}{\abs{\mu}}-\mu\right) x}\geq \delta-\frac{2\delta}{3}= \frac{\delta}{3}.
        \end{equation*}
    \end{itemize}
    Thus, $d(y,\mathbb{F}x)\geq \frac{\delta}{3}$. Using the Hahn--Banach Theorem, we can find a functional $x^*\in S_{E^*}$ such that $x^*(x)=0$ and $x^*(y)=d(y,\mathbb{F}x)$, so the proof is concluded.
\end{proof}

Our next result generalizes \cite[Lemma 4.4]{FOPT} to the complex setting.
\begin{lemma}\label{lemma: SPR embedding of ell infty}
    For any infinite cardinal $\kappa$ there is a $12$-SPR isometric embedding of $\ell_{\infty}(\kappa)$ into itself.
\end{lemma}

\begin{proof}
    Let us denote by $E=\ell_{\infty}(\kappa)$ and $E_*=\ell_1(\kappa)$. Let $(z_\gamma)_{\gamma\in \Gamma}$ be a dense set in $S_{E_*}$ with $\abs{\Gamma}=\kappa$, and define the operator $J:E\rightarrow \ell_{\infty}(\Gamma)$ by $Jx=(x(z_\gamma))_{\gamma\in \Gamma}$. By the density of $(z_\gamma)_{\gamma\in \Gamma}$ it can be checked that $J$ is an isometric embedding. Let us show that $J(E)$ does SPR in $\ell_{\infty}(\Gamma)$, that is, there is some $C>0$ such that for every $x, y\in E$,
    \begin{equation*}
        \min_{\abs{\lambda}=1}\|Jx-\lambda Jy\|\leq C \|\abs{Jx}-\abs{Jy}\|. 
    \end{equation*}
  Without loss of generality, we may assume that $x\in S_E$, $y\in B_E$, and $\min_{\abs{\lambda}=1}\|x-\lambda y\|\geq \frac{1}{\sqrt{2}}$. Indeed, suppose that the inequality above is established (with some constant $C'$) for all $x$ and $y$ that satisfy the aforementioned conditions. Then, given any $f,g\in E$ we can construct vectors $f',g'$ satisfying the properties listed in the statement of \cite[Theorem 3.10]{FOPT}. In particular, denoting $x=\frac{f'}{\|f'\|}$ and $y=\frac{g'}{\|f'\|}$, we get that
    \begin{align*}
        \min_{\abs{\lambda}=1}\|Jf-\lambda Jg\| & \leq \sqrt{2}\min_{\abs{\lambda}=1}\|Jf'-\lambda Jg'\| = \sqrt{2}\|f'\| \min_{\abs{\lambda}=1}\|Jx-\lambda Jy\| \\
        & \leq \sqrt{2}\|f'\| C' \|\abs{Jx}-\abs{Jy}\| = \sqrt{2} C' \|\abs{Jf'}-\abs{Jg'}\| \leq \sqrt{2} C' \|\abs{Jf}-\abs{Jg}\|, 
    \end{align*}
    so $J(E)$ does SPR with constant $C=\sqrt{2}C'$.\medskip

    Now, if $x\in S_E$, $y\in B_E$ and $\min_{\abs{\lambda}=1}\|x-\lambda y\|\geq \frac{1}{\sqrt{2}}$, we can apply \Cref{lemma: separating functional} to obtain $x^*\in E^*$ such that $\abs{\,\abs{x^*(x)}-\abs{x^*(y)}}\geq \frac{1}{3\sqrt{2}}$. Using Goldstine's theorem and the density of $(z_\gamma)_{\gamma\in \Gamma}$ in $E_*$, we may find, for every $\eps>0$, some $z_{\gamma_\varepsilon}$ such that $\abs{x^*(x)-x(z_{\gamma_\varepsilon})}<\eps$ and $\abs{x^*(y)-y(z_{\gamma_\varepsilon})}<\eps$. Thus, $\abs{\,\abs{x(z_{\gamma_\varepsilon})}-\abs{y(z_{\gamma_\varepsilon})}}\geq \frac{1}{3\sqrt{2}}-2\eps$. It follows that
    \begin{equation*}
        \|\abs{Jx}-\abs{Jy}\|\geq \abs{\,\abs{x(z_{\gamma_\varepsilon})}-\abs{y(z_{\gamma_\varepsilon})}}\geq \frac{1}{3\sqrt{2}}-2\eps
    \end{equation*}
    for every $\eps>0$, so that
    \begin{equation*}
        6\sqrt{2} \|\abs{Jx}-\abs{Jy}\|\geq 2\geq \min_{\abs{\lambda}=1}\|Jx-\lambda Jy\| .
    \end{equation*}
    We can therefore take $C'=6\sqrt{2}$ and $C=12$.
\end{proof}

\begin{rem}\label{rem: SPR embedding of E in C(B_E*)}
    The same argument can be used to show that any Banach space $E$ embeds isometrically into $C(B_{E^*})$ as a $12$-SPR subspace by means of the canonical embedding that sends every $x\in E$ to the evaluation function $x^*\mapsto x^*(x)$.
\end{rem}

We now proceed to show that $C[0,1]$ is a universal container for every (complex) separable Banach space.

\begin{prop}\label{prop: SPR embedding of E in C01}
    Every separable Banach space embeds isometrically into $C(\Delta)$ as well as $C[0,1]$ as a $12$-SPR subspace (here $\Delta$ is the Cantor set).
\end{prop}

\begin{proof}
    Given a separable Banach space $E$, let $K:=(B_{E^*},w^*)$ and let $j:E\rightarrow C(K)$ be the isometric embedding introduced in \Cref{rem: SPR embedding of E in C(B_E*)}. Recall that $\|\abs{jx}-\abs{jy}\|\geq \frac{1}{3\sqrt{2}}$ for every $x\in S_E$, $y\in B_E$ such that $\min_{\abs{\lambda}=1}\|x-\lambda y\|\geq \frac{1}{\sqrt{2}}$. Since $E$ is separable, $K$ is a metrizable compact space. Hence, \cite[Theorem 4.18]{Kechris} produces a continuous surjection $h:\Delta\rightarrow K$ that defines a composition operator $H:C(K)\rightarrow C(\Delta)$. In particular, $H$ is a lattice isometric embedding. Now, since $\Delta\subseteq [0,1]$, \cite[Theorem 4.4.4]{AlbiacKalton} provides an extension operator $T:C(\Delta)\rightarrow C[0,1]$ such that $\eval[0]{Tf}_{\Delta}=f$, $T\uno=\uno$, $\|T\|=1$ and $Tf\geq 0$ for every $f\geq 0$ (it is straightforward to check that this result is also valid in the complex setting). In particular, $T$ is an isometric embedding. Moreover, given $x\in S_E$, $y\in B_E$ such that $\min_{\abs{\lambda}=1}\|x-\lambda y\|\geq \frac{1}{\sqrt{2}}$, \Cref{lemma: separating functional} together with the surjectivity of $h$ implies that there is some $t\in \Delta$ such that
    \begin{align*}
        \frac{1}{3\sqrt{2}} & \leq \|\abs{jx}-\abs{jy}\| =  \abs{\,\abs{jx(h(t))}-\abs{jy(h(t))}} = \abs{\,\abs{Hjx(t)}-\abs{Hjy(t)}}    \\
        & =  \abs{\,\abs{THjx(t)}-\abs{THjy(t)}} \leq \|\abs{THjx}-\abs{THjy}\|  .
    \end{align*}
    Repeating the argument used in \Cref{lemma: SPR embedding of ell infty} we conclude that $THj(E)$ does $12$-SPR in $C[0,1]$, and similarly $C(\Delta)$ contains a $12$-SPR isometric copy of $E$.
\end{proof}

\subsection{\texorpdfstring{\Cref{thm:intro-existence-subspace-SPR-C(K)}}{} in the complex setting}

We now turn our attention to seeking conditions on a compact Hausdorff space $K$ that ensure the existence of (infinite-dimensional) SPR subspaces in $C(K)$. Our objective is to extend to the complex setting the following result proved in \cite{FOPT} for real scalars: $K'$ is infinite if and only if $C(K)$ contains an SPR subspace. More precisely, we will show  the following theorem announced in the introduction (\Cref{thm:intro-existence-subspace-SPR-C(K)}).

\begin{theorem}\label{thm: SPR subspace iif K' infinite}
Let $K$ be a compact Hausdorff space. The following statements are equivalent over both the real and complex fields:
    \begin{enumerate}[(i)]
        \item $K'$, the set of non-isolated points of $K$, is infinite.
        \item $c_0$ embeds isometrically into $C(K)$ as an SPR subspace.
        \item $C(K)$ contains an infinite-dimensional SPR subspace.
    \end{enumerate}    
\end{theorem}

We will present the proof of the above result at the end of this subsection, as a direct consequence of several partial results developed throughout the section. One of the most critical of these will be the next technical lemma, which clearly illustrates the additional difficulty inherent in studying \emph{complex} SPR. 

Inspired by the proof scheme in \cite[Theorem 6.1]{FOPT}, where one embeds $c_0$ into a $C(K)$-space (with $K'$ infinite) by creating overlapping points for any pair of natural numbers, we use the characterization of complex SPR from \Cref{thm: complex stable phase retrieval} to prove that whenever an embedding between complex $C(K)$-spaces satisfies a certain overlapping property, then it will perform complex SPR.

\begin{lemma}\label{lemma: suficient condition for complex SPR}
    Let $K,L$ be compact Hausdorff spaces, and assume that the field of scalars is $\C$. Suppose that $T:C(L)\rightarrow C(K)$ is a contraction so that for every pair $(s,t)\in L\times L$ there are $r_{s,t},\widetilde{r}_{s,t}\in K$ such that $Tf(r_{s,t})=\frac{1}{2}(f(s)+f(t))$ and $Tf(\widetilde{r}_{s,t})=\frac{1}{2}(f(s)+if(t))$. Then $T$ is an SPR embedding.
\end{lemma}

\begin{proof}
    The fact that $T$ is an isometry follows from the fact that $T$ is contractive and $Tf(r_{s,s})=f(s)$ for every $s\in L$. We claim that $T(C(L))$ does complex SPR. By \Cref{thm: complex stable phase retrieval}, this is equivalent to showing that there exists some $\delta>0$ such that if $f,g\in S_{C(L)}$ satisfy $\min_{|\lambda|=1}\|f-\lambda g\|\geq \frac{1}{5}$, then $\|\Re{Tf\overline{Tg}}\|\geq \delta$. Fix $f$ and $g$ as above and put $f(s)=A_se^{i\theta_s}$ and $g(s)=B_se^{i\sigma_s}$ for every $s\in L$. Let $0<c\ll 1$ be a small parameter whose value will be established later. It will be useful for the rest of the proof to record the explicit values of $\Re{Tf\overline{Tg}}$ at the points $r_{s,t}$ and $\widetilde{r}_{s,t}$ for $s,t\in L$:
\begin{align*}
    \Re{Tf\overline{Tg}}(r_{s,t})=\frac{1}{4}[&A_sB_s\cos(\theta_s-\sigma_s)+A_tB_t\cos(\theta_t-\sigma_t)\\
    +&A_tB_s\cos(\theta_t-\sigma_s)+A_sB_t\cos(\theta_s-\sigma_t)],
\end{align*}
\begin{align*}
    \Re{Tf\overline{Tg}}(\widetilde{r}_{s,t})=\frac{1}{4}[&A_sB_s\cos(\theta_s-\sigma_s)+A_tB_t\cos(\theta_t-\sigma_t)\\
    -&A_tB_s\sin(\theta_t-\sigma_s)+A_sB_t\sin(\theta_s-\sigma_t)].
\end{align*}
Since both $f$ and $g$ have norm one, there exist some $s,t\in  L$ such that $A_s=1=B_t$. We distinguish two main cases:\medskip

\noindent \textbf{Case A:} $A_t>\frac{9}{10}$ or $B_s>\frac{9}{10}$ (where without loss of generality, we assume the latter).
\begin{itemize}
    \item If $|\cos(\theta_s-\sigma_s)|\geq c$, then 
    \[\|\Re{Tf\overline{Tg}}\|\geq |\Re{Tf\overline{Tg}}(r_{s,s})|=A_sB_s|\cos(\theta_s-\sigma_s)|\geq \frac{9}{10} c.\] 
    \item If $|\cos(\theta_s-\sigma_s)|< c$, then we use the fact that $\min_{|\lambda|=1}\|f-\lambda g\|\geq \frac{1}{5}$ to find some $u\in L$ such that 
    \begin{equation}\label{eq: choice of phase for m}
        \frac{1}{5}\leq |f(u)-e^{i(\theta_s-\sigma_s)}g(u)|\leq A_u+B_u.
    \end{equation} 
    Note that this $u$ must be different from $s$, since $|f(s)-e^{i(\theta_s-\sigma_s)}g(s)|=|A_s-B_s|=|1-B_s|<\frac{1}{10}$. 
    \begin{itemize}
        \item If $B_u<c$, then necessarily $A_u\geq \frac{1}{5}-c$. Since either $|\cos(\theta_u-\sigma_s)|$ or $|\sin(\theta_u-\sigma_s)|$ is greater or equal than $\frac{1}{\sqrt{2}}$, evaluating at $r_{s,u}$ or $\widetilde{r}_{s,u}$, respectively, and applying the triangle inequality, we get that 
        \[4\|\Re{Tf\overline{Tg}}\|\geq \frac{9}{10\sqrt{2}}A_u-3c\geq \frac{9}{50\sqrt{2}}-\intoo[2]{3+\frac{9}{10\sqrt{2}}}c>0.\] 
        \item If $A_u<c$, we proceed analogously.
        \item If $A_u,B_u\geq c$, we distinguish two cases: \begin{itemize}
            \item If $|\cos(\theta_u-\sigma_u)|\geq c$, then 
            \[\|\Re{Tf\overline{Tg}}\|\geq |\Re{Tf\overline{Tg}}(r_{u,u})|=A_uB_u|\cos(\theta_u-\sigma_u)|\geq c^3.\] 
            \item Otherwise, we are in a situation where $A_u,B_u\geq c$ and both $|\cos(\theta_s-\sigma_s)|,|\cos(\theta_u-\sigma_u)|<c$. This case requires a more careful analysis than the previous ones, so we treat it separately.
        \end{itemize}
    \end{itemize}
\end{itemize}

            \noindent Let $\Delta=(\theta_s-\sigma_s)+(\theta_u-\sigma_u)=(\theta_s-\sigma_u)+(\theta_u-\sigma_s)$. Using standard trigonometric relations, it is clear that $|\cos\Delta|\geq 1-2c^2$ and $|\sin\Delta|\leq 2c$. 
            Moreover, we can write
            \[\cos(\theta_s-\sigma_u)=\cos \Delta \cos(\theta_u-\sigma_s)+\sin \Delta \sin (\theta_u-\sigma_s)\]
            and
            \[\sin(\theta_s-\sigma_u)=\sin \Delta \cos(\theta_u-\sigma_s)-\cos \Delta \sin (\theta_u-\sigma_s).\]
            When $\cos\Delta>0$, we have $0\leq 1-\cos \Delta \leq 2c^2$, so that
            \begin{equation}\label{eq: C and S when CDelta positive}
            \begin{aligned}
                |\cos(\theta_s-\sigma_u)-\cos(\theta_u-\sigma_s)|& \leq 4c,\\
                |\sin(\theta_s-\sigma_u)+\sin(\theta_u-\sigma_s)|&\leq 4c.
            \end{aligned}
            \end{equation}
            When $\cos\Delta<0$, it follows that $0\leq 1+\cos \Delta \leq 2c^2$, so instead we get
            \begin{equation}\label{eq: C and S when CDelta negative}
            \begin{aligned}
                |\cos(\theta_s-\sigma_u)+\cos(\theta_u-\sigma_s)|& \leq 4c,\\
                |\sin(\theta_s-\sigma_u)-\sin(\theta_u-\sigma_s)|&\leq 4c.
            \end{aligned}
            \end{equation}
            According to this dichotomy, we have the following conclusions:
            \begin{itemize}
                \item If $\cos\Delta>0$, then evaluating at $r_{s,u}$, using the equations \eqref{eq: choice of phase for m}, \eqref{eq: C and S when CDelta positive} and invoking the assumption $B_s>\frac{9}{10}$ we obtain that
                \begin{align*}
                     4|\Re{Tf\overline{Tg}}(r_{s,u})|&\geq |B_u\cos(\theta_s-\sigma_u)+A_uB_s\cos(\theta_u-\sigma_s)|-2c\\
                    & \geq (A_u+B_u)|\cos(\theta_s-\sigma_u)|-A_u(1-B_s)|\cos(\theta_u-\sigma_s)|\\
                    &-A_u|\cos(\theta_u-\sigma_s)-\cos(\theta_s-\sigma_u)|-2c\\
                    & \geq \frac{1}{5}|\cos(\theta_s-\sigma_u)|-\frac{1}{10}-6c,
                \end{align*}
                while at $\widetilde{r}_{s,u}$ we get
                \begin{align*}
                    4|\Re{Tf\overline{Tg}}(\widetilde{r}_{s,u})| &\geq |B_u\sin(\theta_s-\sigma_u)-A_uB_s\sin(\theta_u-\sigma_s)|-2c\\
                    & \geq (A_u+B_u)|\sin(\theta_s-\sigma_u)|-A_u(1-B_s)|\sin(\theta_u-\sigma_s)|\\
                    &-A_u|\sin(\theta_u-\sigma_s)+\sin(\theta_s-\sigma_u)|-2c\\
                    & \geq \frac{1}{5}|\sin(\theta_s-\sigma_u)|-\frac{1}{10}-6c.
                \end{align*}
                Therefore,
                \[4\|\Re{Tf\overline{Tg}}\|\geq \frac{1}{5}\intoo[2]{\frac{1}{\sqrt{2}}-\frac{1}{2}}-6c>0.\]
                \item If $\cos\Delta<0$, then again by standard trigonometric relations it follows that
                \begin{equation}\label{eq: cos difference of phases}
                    \cos((\theta_s-\sigma_s)-(\theta_u-\sigma_u))=\cos(\Delta-2(\theta_u-\sigma_u))\geq 1-6c^2.
                \end{equation}
                We now split one last time into two cases: 
                \begin{itemize}
                    \item If $|A_u-B_u|\geq \frac{\sqrt{24}}{25}$, we can evaluate at $r_{s,u}$ and $\widetilde{r}_{s,u}$ and apply \eqref{eq: C and S when CDelta negative} to obtain
                    \begin{align*}
                        4|\Re{Tf\overline{Tg}}(r_{s,u})|&\geq |B_u\cos(\theta_s-\sigma_u)+A_uB_s\cos(\theta_u-\sigma_s)|-2c\\
                        & \geq |A_u-B_u||\cos(\theta_s-\sigma_u)|-A_u(1-B_s)|\cos(\theta_u-\sigma_s)|\\
                        &-A_u|\cos(\theta_u-\sigma_s)+\cos(\theta_s-\sigma_u)|-2c\\
                        & \geq \frac{\sqrt{24}}{25}|\cos(\theta_s-\sigma_u)|-\frac{1}{10}-6c,
                    \end{align*}
                    and
                    \begin{align*}
                        4|\Re{Tf\overline{Tg}}(\widetilde{r}_{s,u})&|\geq |B_u\sin(\theta_s-\sigma_u)-A_uB_s\sin(\theta_u-\sigma_s)|-2c\\
                        & \geq |A_u-B_u||\sin(\theta_s-\sigma_u)|-A_u(1-B_s)|\sin(\theta_u-\sigma_s)|\\
                        &-A_u|\sin(\theta_u-\sigma_s)-\sin(\theta_s-\sigma_u)|-2c\\
                        & \geq \frac{\sqrt{24}}{25}|\sin(\theta_s-\sigma_u)|-\frac{1}{10}-6c,
                    \end{align*}
                    respectively. Again, this implies that
                    \[4\|\Re{Tf\overline{Tg}}\|\geq \frac{1}{5}\intoo[2]{\frac{\sqrt{12}}{5}-\frac{1}{2}}-6c>0.\]
                    \item On the contrary, if $|A_u-B_u|< \frac{\sqrt{24}}{25}$, by \eqref{eq: choice of phase for m} and \eqref{eq: cos difference of phases} we have
                    \begin{align*}
                        \frac{1}{25}&\leq |A_u-B_ue^{i((\theta_s-\sigma_s)-(\theta_u-\sigma_u))}|^2\\
                        & = A_u^2+B_u^2 -2A_uB_u\cos((\theta_s-\sigma_s)-(\theta_u-\sigma_u))\\
                        & =|A_u-B_u|^2+2A_uB_u[1-\cos((\theta_s-\sigma_s)-(\theta_u-\sigma_u))]\\
                        & < \frac{24}{25^2}+12c^2<\frac{1}{25}, 
                    \end{align*}
                    which becomes a contradiction as long as we choose $c<\frac{1}{25\sqrt{12}}$. Therefore, this situation is not possible, and we have exhausted all of the possibilities within Case A.
                \end{itemize}
            \end{itemize}

\noindent Note that by choosing $0<c<\frac{1}{30}(\frac{\sqrt{12}}{5}-\frac{1}{2})$ and $0<\delta\leq\min\{ c^3,\frac{1}{4}(\frac{1}{5}(\frac{\sqrt{12}}{5}-\frac{1}{2})-6c)\}$, all the conditions that have appeared in Case A are satisfied.\medskip

\noindent \textbf{Case B:} $A_t<\frac{9}{10}$ and $B_s<\frac{9}{10}$ (in particular, $s\neq t$).
\begin{itemize}
    \item If $B_s<c$, we examine $A_t$:
    \begin{itemize}
        \item If $A_t<c$, then 
        \[4|\Re{Tf\overline{Tg}}(r_{s,t})|\geq |\cos(\theta_s-\sigma_t)|-3c\]
        and
        \[4|\Re{Tf\overline{Tg}}(\widetilde{r}_{s,t})|\geq |\sin(\theta_s-\sigma_t)|-3c,\]
        so
        \[4\|\Re{Tf\overline{Tg}}\|\geq \frac{1}{\sqrt{2}}-3c>0.\]
        \item If $A_t\geq c$, we split again:
        \begin{itemize}
            \item If $|\cos(\theta_t-\sigma_t)|\geq c$, then 
            \[\|\Re{Tf\overline{Tg}}\|\geq |\Re{Tf\overline{Tg}}(r_{t,t})|\geq c^2.\] 
            \item If $|\cos(\theta_t-\sigma_t)|< c$, then evaluating again at either $r_{s,t}$ or $\widetilde{r}_{s,t}$ we obtain that 
            \[4\|\Re{Tf\overline{Tg}}\|\geq \frac{1}{\sqrt{2}}-3c>0.\]
        \end{itemize}
    \end{itemize}
    \item If $B_s\geq c$, we look at $|\cos(\theta_s-\sigma_s)|$:
    \begin{itemize}
        \item If $|\cos(\theta_s-\sigma_s)|\geq c$, it follows that
        \[\|\Re{Tf\overline{Tg}}\|\geq |\Re{Tf\overline{Tg}}(r_{s,s})|\geq c^2.\] 
        \item If $|\cos(\theta_s-\sigma_s)|< c$, we split again:
        \begin{itemize}
            \item If $A_t<c$, evaluating at the points $r_{s,t}$ and $\widetilde{r}_{s,t}$ yields 
            \[4\|\Re{Tf\overline{Tg}}\|\geq \frac{1}{\sqrt{2}}-3c>0.\]
            \item If $A_t\geq c$, then we examine $|\cos(\theta_t-\sigma_t)|$:
            \begin{itemize}
                \item  If $|\cos(\theta_t-\sigma_t)|\geq c$, then
                \[\|\Re{Tf\overline{Tg}}\|\geq |\Re{Tf\overline{Tg}}(r_{t,t})|\geq c^2.\] 
                \item If $|\cos(\theta_t-\sigma_t)|< c$, we arrive at a situation similar to Case A, but not exactly the same, since we might not have an analogue of the condition \eqref{eq: choice of phase for m}. Therefore, to deal with the case $A_t,B_s\geq c$ and $|\cos(\theta_s-\sigma_s)|,|\cos(\theta_t-\sigma_t)|<c$, we proceed as follows:
            \end{itemize}
        \end{itemize}
    \end{itemize}
\end{itemize}

\noindent Let $\Gamma=(\theta_s-\sigma_s)+(\theta_t-\sigma_t)$ and observe that $|\cos\Gamma|\geq 1-2c^2$. We can again obtain the estimates \eqref{eq: C and S when CDelta positive} and \eqref{eq: C and S when CDelta negative}, with the obvious modifications. We now split the analysis as in Case~A:
\begin{itemize}
    \item If $\cos\Gamma>0$, then if we evaluate at $r_{s,t}$ we obtain that
    \begin{align*}
         4|\Re{Tf\overline{Tg}}(r_{s,t})|&\geq |\cos(\theta_s-\sigma_t)+A_tB_s\cos(\theta_t-\sigma_s)|-2c\\
        & \geq (1+A_tB_s)|\cos(\theta_s-\sigma_t)|-A_tB_s|\cos(\theta_t-\sigma_s)-\cos(\theta_s-\sigma_t)|-2c\\
        & \geq |\cos(\theta_s-\sigma_t)|-6c,
    \end{align*}
    while at $\widetilde{r}_{s,t}$ we get
    \begin{align*}
         4|\Re{Tf\overline{Tg}}(\widetilde{r}_{s,t})|&\geq |\sin(\theta_s-\sigma_t)-A_tB_s\sin(\theta_t-\sigma_s)|-2c\\
        & \geq (1+A_tB_s)|\sin(\theta_s-\sigma_t)|-A_tB_s|\sin(\theta_t-\sigma_s)+\sin(\theta_s-\sigma_t)|-2c\\
        & \geq |\sin(\theta_s-\sigma_t)|-6c.
    \end{align*}
    Therefore, 
    \[4\|\Re{Tf\overline{Tg}}\|\geq \frac{1}{\sqrt{2}}-6c>0.\] 
    \item If $\cos\Gamma<0$, then we use that $A_t,B_s< \frac{9}{10}$ to see that $1-A_tB_s> \frac{1}{10}$ and then apply the analogue of \eqref{eq: C and S when CDelta negative} to obtain
        \begin{align*}
            4|\Re{Tf\overline{Tg}}(r_{s,t})|&\geq |\cos(\theta_s-\sigma_t)+A_tB_s\cos(\theta_t-\sigma_s)|-2c\\
            & \geq (1-A_tB_s)|\cos(\theta_s-\sigma_t)|-A_tB_s|\cos(\theta_t-\sigma_s)+\cos(\theta_s-\sigma_t)|-2c\\
            & \geq \frac{1}{10}|\cos(\theta_s-\sigma_t)|-6c,
        \end{align*}
        and
        \begin{align*}
            4|\Re{Tf\overline{Tg}}(\widetilde{r}_{s,t})|&\geq |\sin(\theta_s-\sigma_t)-A_tB_s\sin(\theta_t-\sigma_s)|-2c\\
            & \geq (1-A_tB_s)|\sin(\theta_s-\sigma_t)|-A_tB_s|\sin(\theta_t-\sigma_s)-\sin(\theta_s-\sigma_t)|-2c\\
            & \geq \frac{1}{10}|\sin(\theta_s-\sigma_t)|-6c,
        \end{align*}
        and therefore
        \[4\|\Re{Tf\overline{Tg}}\|\geq \frac{1}{10\sqrt{2}}-6c>0.\]
\end{itemize}
Observe that Case B has provided new restrictions for $c$ and $\delta$, but it can be checked that the choice of parameters made in Case A satisfies these new conditions. Hence, we have shown that there is a $\delta>0$ such that $\|\Re{Tf\overline{Tg}}\|\geq \delta$ whenever $f,g\in S_{C(L)}$ satisfy $\min_{|\lambda|=1}\|f-\lambda g\|\geq \frac{1}{5}$, so the proof is concluded.
\end{proof}

\begin{rem}\label{rem: relaxed suficient condition for complex SPR}
    It should be noted that the proof of \Cref{lemma: suficient condition for complex SPR} works with a \emph{weaker assumption}: namely, that for every point $s\in L$ there exists $r_s\in K$ such that $Tf(r_{s})=f(s)$ and for every pair of distinct points $s,t\in L$ there are $r_{s,t},\widetilde{r}_{s,t}\in K$ such that $Tf(r_{s,t})=\frac{1}{2}(f(s)+f(t))$ and $Tf(\widetilde{r}_{s,t})$ is either $\frac{1}{2}(f(s)+if(t))$ or $\frac{1}{2}(f(t)+if(s))$. 
\end{rem}

Next, we show that $C_0[1,\omega^2)$ (which can be identified with the closed sublattice of $C[1,\omega^2]$ consisting of functions that vanish at $\omega^2$) contains an isometric copy of $c_0$ doing SPR. The construction for the complex case is an adaptation of the real one, extracted from \cite{GH}, so that it satisfies the overlapping condition of \Cref{lemma: suficient condition for complex SPR}. Throughout the proof, $\uno_{m_1}$ will represent the characteristic function of the clopen set $(\omega\cdot(m_1-1),\omega\cdot m_1]$ for any natural number $m_1\geq 1$. Analogously, given natural numbers $m_1,m_2\geq 1$, $\uno_{m_1,m_2}$ will stand for the characteristic function of the singleton $\{\omega\cdot (m_1-1)+m_2\}$.

\begin{prop}\label{prop: SPR copy of c0 in C0[0 w2]}
    Assume that the field of scalars is $\C$. There exists an SPR embedding of $c_0$ into $C_0[1,\omega^2)$.
\end{prop}

\begin{proof}
For each $n\in \N$, we define the expression
\[x^{(n)}=\uno_{1,n}+\frac{1}{2}\uno_{n+1}+\frac{1}{2}\sum_{m=2}^{n}(\uno_{m,2n-1}+i\uno_{m,2n})\]
(see \Cref{Fig: complex SPR copy of c0}), which is continuous on $[1,\omega^2]$ and vanishes at $\omega^2$.

\begin{figure}[!htb]
\centering
\includegraphics[scale=1.8]{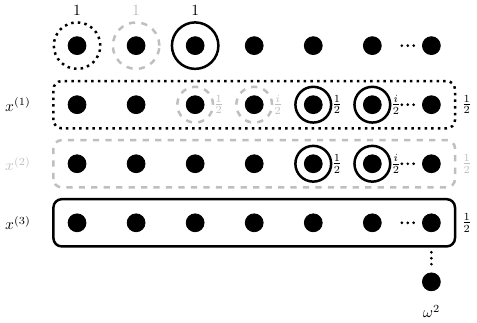}
\caption{\label{Fig: complex SPR copy of c0} Representation of $x^{(n)}$. The points in $[1,\omega^2]$ are ordered from left to right and from the top to the bottom.}
\end{figure}
The fact that the sequence $(x^{(n)})_n$ spans an isometric copy of $c_0$ can be checked in an analogous way to what was done in \cite[Proposition 3.8]{GH} and \cite[Theorem 6.1]{FOPT}. Indeed, fix any finite sequence of complex scalars $(\alpha_k)_{k=1}^n$ with $\max_{1\leq k\leq n} |\alpha_k|=1$ and let us check that $x=\sum_{k=1}^n\alpha_kx^{(k)}$ has norm $1$ in $C_0[1,\omega^2)$. Since $x(m)=\alpha_m$ for any $1\leq m\leq n$, then $\|x\|\geq 1$. For the reverse inequality, observe that
\begin{itemize}
    \item $x$ is zero on $[n+1,\omega]\cup (\omega\cdot(n+1),\omega^2]$.
    \item $x$ is $\frac{1}{2}\alpha_k$ on $(\omega\cdot k,\omega\cdot k+2k]\cup (\omega\cdot k+2n,\omega\cdot(k+1)]$, for every $1\leq k\leq n$.
    \item $x$ is $\frac{1}{2}\alpha_k+\frac{1}{2}\alpha_j$ on the isolated point $\{\omega\cdot k+2j-1\}$ and $\frac{1}{2}\alpha_k+\frac{1}{2}i\alpha_j$ on the isolated point $\{\omega\cdot k+2j\}$, for all $1\leq k<j\leq n$.
\end{itemize}

Moreover, it is clear that if we take $r_n=n$ for every natural number $n\in \N$, and $r_{m,n}=\omega \cdot m+2n-1$ and $\widetilde{r}_{m,n}=\omega \cdot m+2n$ for every pair $m<n$, then the relaxed condition from \Cref{rem: relaxed suficient condition for complex SPR} holds, so \Cref{lemma: suficient condition for complex SPR} guarantees that $E=\overline{\spn}\{x^{(n)}:n\in\N\}$ does complex SPR (the fact that $c_0$ is not a $C(K)$-space, but only an AM-space, is not relevant in the proof).
\end{proof}

Recall that a \textit{perfect set} $P$ of a topological space $T$ is a non-empty closed set which has no isolated points, and a topological space $T$ is said to be \textit{scattered} if it does not contain any perfect sets. In \cite[Proposition 3.5]{GH}, it is shown that if $K'$ is infinite, there there exists a \emph{nice} isometric embedding $T$ from $C_0[1,\omega^2)$ (in fact, $C[1,\omega^2])$ into $C(K)$. In general, we cannot ensure that this embedding is a lattice homomorphism \cite[Remark 3.6]{GH}. However, in the case when $K$ is scattered, we can achieve this, as we show next.

\begin{lemma}\label{lem: scattered K' infinite}
Let $K$ be a scattered compact Hausdorff space. If $K'$ is infinite, then $C[1,\omega^2]$ embeds lattice isometrically into $C(K)$. In particular, $C(K)$ contains an isometric SPR  copy of $c_0$. 
\end{lemma}

\begin{proof}
    If $K$ is scattered, then it is zero-dimensional, i.e., the topology admits a base of clopens (cf.~\cite[Theorem 8.5.4]{Semadeni}). Now, since $K'$ is infinite, or equivalently, $K''\neq \varnothing$, by \cite[Theorem 4.2]{RS} we have that $C[1,\omega^2]$ embeds lattice isometrically into $C(K)$.
\end{proof}

In the other extreme case, we obtain the following:

\begin{theorem}\label{thm: perfect set}
Let $K$ be a compact Hausdorff space. Then $K$ contains a perfect set if and only if $C(K)$ contains an isometric SPR copy of every separable Banach space.
\end{theorem}
\begin{proof}
Suppose that $K$ contains a perfect set. Then, by \cite[Theorem 2, Section 4, p.~29]{Lacey}, $C[0,1]$ embeds lattice isometrically into $C(K)$ by means of a composition operator associated to a continuous surjection $\phi:K\rightarrow [0,1]$. Since, by \Cref{prop: SPR embedding of E in C01}, $C[0,1]$ contains an isometric SPR copy of every separable Banach space, $C(K)$ does as well (observe that lattice embeddings preserve SPR subspaces).
\medskip

Conversely, assume that $C(K)$ contains an isometric SPR copy of every separable Banach space. In particular, $C(K)$ contains an isomorphic copy of $\ell_1$ and hence, by \cite[Corollary to Theorem 4, Section 13, p.~113]{Lacey}, $K$ is not scattered, that is, $K$ contains a perfect set.
\end{proof}

\begin{proof}[Proof of \Cref{thm: SPR subspace iif K' infinite}]
$(i)\Rightarrow(ii)$ Suppose that $K'$ is infinite. We distinguish between two cases:
\begin{itemize}
    \item If $K$ contains a perfect set, then, by \Cref{thm: perfect set}, $C(K)$ contains an isometric SPR copy of every separable Banach space and, in particular, of $c_0$.
    \item If $K$ is scattered, then, by \Cref{lem: scattered K' infinite}, $C[1,\omega^2]$ embeds into $C(K)$ in a lattice isometric way, and hence $c_0$ isometrically embeds into $C(K)$ doing SPR.
\end{itemize}

 Implication $(ii)\Rightarrow(iii)$ is trivial, so it remains to show $(iii)\Rightarrow(i)$. By mimicking the gliding hump argument detailed in \cite[Proposition 2.2]{GH}, it is not difficult to check that if $K'$ is finite and $E$ is an infinite-dimensional subspace of $C(K)$, then for every $\varepsilon>0$ we can find $f,g\in S_E$ such that $\||f|\land |g|\|\leq\varepsilon$. By \cite[Proposition 2.5]{CGH}, in this situation, $E$ fails SPR.
\end{proof}
\begin{rem}
    Let $X$ be a Banach lattice. An interesting consequence of \Cref{thm:intro-existence-subspace-SPR-C(K)}  is that the vector-valued function space $C(K;X)$ may admit infinite-dimensional SPR subspaces even when the individual components $C(K)$ and $X$ do not. Indeed, one may take $K=[0,\omega]$, that is, the one point compactification of the natural numbers, and $X=C[0,\omega]$. It is easy to see that if $K$ and $L$ are compact spaces, then $C(K,C(L))$ is lattice isometric to $C(K\times L)=C(K)\otimes_{|\pi|}C(L)$. In particular, $C([0,\omega]; C[0,\omega])$ can be identified with $C([0,\omega]^2)$. Clearly, $[0,\omega]$ has only one accumulation point ($\omega$ itself), so $C[0,\omega]$ does not have SPR subspaces. However, $[0,\omega]^2$ has an infinite number of accumulation points, so $C([0,\omega]^2)=C([0,\omega]; C[0,\omega])$ contains SPR subspaces.
\end{rem}

\section{Quantitative SPR embeddings of \texorpdfstring{$C(K)$}{}-spaces} \label{sect: Timur's question}

After \Cref{thm: SPR subspace iif K' infinite} was established in \cite{FOPT} (for real scalars), the authors asked whether a \textit{quantitative version} of the result was possible:

\begin{question}[Question 6.4 in \cite{FOPT}]\label{question:timur's-question}
    The proof of \cite[Theorem 6.1]{FOPT} shows that $K'$ is infinite if and only if $C(K)$ contains an SPR copy of $c_0$. If $K$ is ``large'' enough (in terms of the smallest ordinal $\alpha$ for which $K^{(\alpha)}$ is finite), what SPR subspaces (other than $c_0$) does $C(K)$ have? Note that $c_0$ is isomorphic to $c = C[0, \omega]$ ($\omega$ is the first infinite ordinal). If $K^{(\alpha)}$ is infinite, does $C(K)$ contain an SPR copy of $C[0, \omega^\alpha]$? 
\end{question}

In this section, we aim to address this question by providing an \emph{affirmative answer} in both the real and complex settings. We will show that, for any countable ordinal $1\leq\alpha<\omega_1$, if $K^{(\alpha\odot 2)}$ is non-empty ($|K^{(\alpha\odot 2)}|\geq 2$ in the complex case), then $C(K)$ contains a subspace isometric to $C[1,\omega^\alpha]$ doing SPR (Theorems \ref{thm: real SPR copy of omega^alpha in omega^alpha2} and \ref{thm: complex SPR copy of omega^alpha in 2omega^alpha2}). This continues the work initiated in \cite{GH}, where the authors only consider the case of finite ordinals and manage to prove that for $2<\alpha<\omega$, $K^{(\alpha)}\neq\varnothing$ implies the existence of an isometric SPR embedding of $C[1,\omega^\alpha]$ into $C(K)$. However, the latter result did not provide new information regarding \Cref{question:timur's-question} as the spaces $C[1,\omega^\alpha]$ and $c_0$ are isomorphic for any finite ordinal $\alpha\geq 1$. 

Nonetheless, the techniques employed in \cite{GH} provide the foundation for the work we will develop in this section. 
Namely, \cite[Proposition 3.5]{GH} will enable us to study \Cref{question:timur's-question} only in spaces of the form $C[1,\omega^\alpha]$, which are easier to handle than general $C(K)$-spaces. Most of the results in this section hold in both the real and the complex setting, so we will assume that the field of scalars is $\R$ or $\C$, unless stated otherwise.

It was established in \cite[Proposition 3.5]{GH} that, for any $\alpha<\omega_1$, if $K^{(\alpha)}\neq\varnothing$, then there exists a \emph{nice} isometric embedding $T:C[1,\omega^\alpha]\to C(K)$, in the sense that it \emph{preserves SPR embeddings}. The key property that ensures that the constructed embedding preserves SPR subspaces is the following:

   \begin{enumerate}[(\textasteriskcentered)]
        \item  for every $t\in [1,\omega^\alpha]$ there is $s_t\in K$ such that $Tf(s_t)=f(t)$ for every $f\in C[1,\omega^\alpha]$.
    \end{enumerate}

It is clear that the above condition guarantees the preservation of overlaps between any pair of functions. Thus, it can be easily verified that if $E$ does $C$-SPR in $C[1,\omega^\alpha]$, then $T(E)$ does $C$-SPR in $C(K)$ \cite[Proposition 3.1]{GH}:

\begin{prop}\label{prop: Key-propertySPR}
    Let $K,L$ be compact Hausdorff spaces and $T:C(L)\to C(K)$ an operator with the following property: 
    \begin{equation}
        \text{\textit{for every }} t\in L \text{\textit{ there is }}s_t\in K \text{\textit{ such that }}Tf(s_t)=f(t) \text{\textit{ for every }}f\in C(L). \label{property *}\tag{\textasteriskcentered} 
    \end{equation}
    Then, if $E$ is an SPR subspace of $C(L)$, $T(E)$ is an SPR subspace of $C(K)$.
\end{prop}

As previously mentioned, \cite[Proposition 3.5]{GH} established that we can embed $C[1,\omega^\alpha]$ into $ C(K)$ with a positive embedding satisfying property \eqref{property *} as long as $K^{(\alpha)}\neq\varnothing$:

\begin{prop}\label{prop: implication K_alpha copy of Comegaalpha}
 Assume that the field of scalars is $\R$ or $\C$. Let $K$ be a compact Hausdorff space, $\alpha<\omega_1$ a countable ordinal, and $m\in \N$. If $|K^{(\alpha)}|\geq m$, then there exists a positive isometric embedding $T:C[1,\omega^\alpha\cdot m]\to C(K)$ with property \eqref{property *}.
\end{prop}

In view of \Cref{lem: scattered K' infinite}, one might also ask whether the embedding $T$ can be made to preserve lattice operations, since it is clear that lattice embeddings preserve SPR subspaces. By \cite[Theorem 4.2]{RS}, if $L$ is metric and $K$ is zero-dimensional, then we can achieve a lattice isometric embedding. However, in general, it is not possible, as \cite[Remark 3.6]{GH} shows.

\subsection{A solution to the Question of Freeman et al.~for countable ordinals}\label{sec: arbitrary ordinals}

Cantor's normal form allows us to express uniquely every ordinal $0<\alpha<\omega_1$ as a finite combination of ordinals of the form $\omega^\beta$ with non-zero natural coefficients. Since $\omega^\beta\cdot 0=0$ for every $\beta$, the representation \eqref{eq: cantor normal form} can be understood as a bijection from $\omega_1$ into $c_{00}(\omega_1,\Z_+):=\bigl\{x:\omega_1\to\Z_+\::\: |\{\gamma\::\: x(\gamma)\neq 0\}|<\aleph_0\bigr\}$, that is, the set of sequences of non-negative integers indexed in $\omega_1$ with finite support. Observe that, in general, the coordinate-wise sum in the set $c_{00}(\omega_1,\Z_+)$ does not coincide with the usual sum of ordinal numbers. Given $\alpha,\beta\in \omega_1$, we denote by $\alpha\oplus \beta$ the ordinal obtained by adding the normal forms of $\alpha$ and $\beta$ in $c_{00}(\omega_1,\Z_+)$. This operation $\oplus:\omega_1\times \omega_1\rightarrow \omega_1$ is associative, symmetric (in contrast to the usual sum of ordinals) and compatible with the order (if $\alpha\leq \beta$, then $\alpha\oplus \gamma\leq \beta\oplus \gamma$ for every $\gamma\in \omega_1$). Moreover, it satisfies that $\alpha\oplus 1=\alpha+1$. \medskip

We will use this operation to show that for every countable ordinal $\alpha$, $C[1,\omega^\alpha]$ embeds isometrically into $C[1,\omega^{\alpha\odot 2}]$ doing SPR in the real setting, and $C[1,\omega^\alpha]$ embeds isometrically into $C[1,\omega^{\alpha\odot 2}\cdot 2]$ doing SPR in the complex setting, where $\alpha\odot 2=\alpha\oplus \alpha$, with a constant independent of $\alpha$. Our first step is the following construction:

\begin{lemma}\label{lemma: construction of map U}
    For every pair of ordinals $\alpha, \beta\in\omega_1$ there is a map $U_{\alpha,\beta}: C[1,\omega^\alpha]\times C[1,\omega^\beta]\rightarrow C[1,\omega^{\alpha\oplus\beta}]$ satisfying the following properties:
\begin{enumerate}[(i)]
    \item for every $f_1,f_2\in C[1,\omega^\alpha]$ and $g_1,g_2\in C[1,\omega^\beta]$,
    \[U_{\alpha,\beta}(f_1+f_2,g_1+g_2)=U_{\alpha,\beta}(f_1,g_1)+U_{\alpha,\beta}(f_2,g_2);\]
    \item for every $f\in C[1,\omega^\alpha]$, $g\in C[1,\omega^\beta]$ and $\lambda \in \R$,
    \[U_{\alpha,\beta}(\lambda f,\lambda g)=\lambda U_{\alpha,\beta}(f,g);\]
    \item for every $f\in C[1,\omega^\alpha]$ and $g\in C[1,\omega^\beta]$, 
    \[\|U_{\alpha,\beta}(f,g)\|_\infty \leq \frac{1}{2}\bigl(\|f\|_\infty+\|g\|_\infty\bigr);\]
    \item for every $s\in [1,\omega^\alpha]$ and $t\in [1,\omega^\beta]$ there exists $r_{s,t}\in [1,\omega^{\alpha\oplus\beta}]$ such that
    \[[U_{\alpha,\beta}(f,g)](r_{s,t})=\frac{1}{2}\bigl(f(s)+g(t)\bigr).\]
    \item for every $r\in [1,\omega^{\alpha\oplus\beta}]$ there exists $s_r\in [1,\omega^\alpha]$ and $t_r\in [1,\omega^\beta]$ such that
    \[[U_{\alpha,\beta}(f,g)](r)=\frac{1}{2}\bigl(f(s_r)+g(t_r)\bigr).\]
\end{enumerate}
\end{lemma}

\begin{proof}
    First of all, we observe that for any pair $0\leq \alpha<\beta<\omega_1$, the construction of the map $U_{\beta,\alpha}$ is automatic once the existence of the map $U_{\alpha,\beta}$ is established. Indeed, for $f\in C[1,\omega^\beta]$ and $g\in C[1,\omega^\alpha]$ we define 
    \[U_{\beta,\alpha}(f,g):=U_{\alpha,\beta}(g,f).\]
    The existence of the map $U_{\alpha,\beta}$ for every $\alpha \leq \beta$ is proved by means of two nested transfinite induction arguments, one for $\alpha\in \omega_1$ and another one for $\beta\geq \alpha$:\medskip
    
    \noindent {\bf Inductive Hypothesis (IH):} Given $\alpha \in \omega_1$, if $U_{\gamma,\delta}$ is defined for every $\gamma<\alpha$ and $\delta\geq \gamma$, then $U_{\alpha,\beta}$ exists for every $\beta\geq \alpha$.\medskip

    \begin{enumerate}[1.]
        \item {\bf Base case ($\alpha=0$):} Given $\beta\geq 0$, it is straightforward to check that
        \[\fullfunction{U_{0,\beta}}{C(\{1\})\times C[1,\omega^\beta]}{C[1,\omega^\beta]}{(f,g)}{\frac{1}{2}(f(1)+g)}\]
        is well-defined and satisfies the properties of the statement.
        \item {\bf Inductive step ($\alpha\to \alpha+1$):} Assume that (IH) has been established for $\alpha$. Let us show that it also holds for $\alpha+1$ by a transfinite induction argument on $\beta\geq \alpha+1$:\medskip

        \noindent {\bf Inductive Hypothesis (IH-I):} Fix $\alpha \in \omega_1$, and assume that $U_{\gamma,\delta}$ has been constructed for every $\gamma\leq\alpha$ and $\delta\geq \gamma$. Let $\beta\geq \alpha+1$. If $U_{\alpha+1,\gamma}$ is defined for every $\alpha+1\leq \gamma <\beta$, then $U_{\alpha+1,\beta}$ exists.\medskip
        
        \begin{enumerate}[2.1)]
            \item {\bf Base case ($\beta=\alpha+1$):} We construct $U_{\alpha+1,\alpha+1}: C[1,\omega^{\alpha+1}]\times C[1,\omega^{\alpha+1}]\rightarrow C[1,\omega^{(\alpha+1)\odot 2}]$ in the following way (the order-preserving homeomorphisms that transform the clopen subsets of $[1,\omega^\gamma]$ into intervals of the form $[1,\omega^\delta]$ for a suitable $\delta$ by means of a translation are not written explicitly):\
            \begin{align*}
                U_{\alpha+1,\alpha+1}(f,g)&|_{(\omega^{\alpha\odot 2+1}\cdot(2n-2),\omega^{\alpha\odot 2+1}\cdot(2n-1)]} =\, U_{\alpha+1,\alpha}\bigl(f|_{(\omega^\alpha\cdot(n-1),\omega^{\alpha+1}]},g|_{(\omega^\alpha\cdot(n-1),\omega^{\alpha}\cdot n]}\bigr), n\in \N;\\
                U_{\alpha+1,\alpha+1}(f,g)&|_{(\omega^{\alpha\odot 2+1}\cdot(2n-1),\omega^{\alpha\odot 2+1}\cdot 2n]}  =\, U_{\alpha,\alpha+1}\bigl(f|_{(\omega^\alpha\cdot(n-1),\omega^{\alpha}\cdot n]},g|_{(\omega^\alpha\cdot n,\omega^{\alpha+1}]}\bigr) , n\in \N;\\
                U_{\alpha+1,\alpha+1}(f,g)&|_{\{\omega^{(\alpha+1)\odot 2}\}}  =\,\frac{1}{2}\bigl(f(\omega^{\alpha+1})+g(\omega^{\alpha+1})\bigr).
            \end{align*}
            Clearly, the function defined above satisfies conditions (i), (ii), (iii) and (v). Let us check that it is continuous on $[1,\omega^{(\alpha+1)\odot 2}]$. By the assumption, it is continuous on each clopen set $(\omega^{\alpha\odot 2+1}\cdot(m-1),\omega^{\alpha\odot 2+1}\cdot m]$ with $m\in \N$. Moreover, it is continuous on $\omega^{(\alpha+1)\odot 2}$: for every $\eps>0$ there exists $n_0\in \N$ such that $|f(\omega^{\alpha+1})-f(s)|<\eps$ and $|g(\omega^{\alpha+1})-g(t)|<\eps$ for every $s,t \in (\omega^{\alpha}\cdot n_0,\omega^{\alpha+1}]$, so using (v) we know that for every $r\in (\omega^{\alpha\odot 2+1}\cdot 2n_0,\omega^{(\alpha+1)\odot 2}]$ there are $s_r,t_r \in (\omega^{\alpha}\cdot n_0,\omega^{\alpha+1}]$ such that $[U_{\alpha,\beta}(f,g)](r)=\frac{1}{2}\bigl(f(s_r)+g(t_r)\bigr)$, so we conclude that
            \[\left|[U_{\alpha+1,\alpha+1}(f,g)](\omega^{(\alpha+1)\odot 2})-[U_{\alpha+1,\alpha+1}(f,g)](r)\right|<\eps.\]
            In other words, $U_{\alpha+1,\alpha+1}(f,g)$ is continuous on $\omega^{(\alpha+1)\odot 2}$ because of property (v) and the fact that for every $n\in \N$, $U_{\alpha+1,\alpha+1}(f,g)$ restricted to $(\omega^{\alpha\odot 2+1}\cdot 2n,\omega^{(\alpha+1)\odot 2}]$ depends only on $f|_{(\omega^\alpha n,\omega^{\alpha+1}]}$ and $g|_{(\omega^\alpha n,\omega^{\alpha+1}]}$. Finally, property (iv) follows from the fact that $U_{\alpha+1,\alpha+1}$ and $U_{\alpha,\alpha+1}$ also satisfy (iv), and $[1,\omega^{(\alpha+1)\odot 2}]^2$ can be decomposed as
            \begin{align*}
                & [1,\omega^{(\alpha+1)\odot 2}]^2 = \intoo[3]{\bigcup_{n=1}^\infty (\omega^{\alpha}\cdot (n-1),\omega^{\alpha+1}]\times (\omega^{\alpha}\cdot (n-1),\omega^{\alpha}\cdot n]}\\
                & \cup \intoo[3]{\bigcup_{n=1}^\infty (\omega^{\alpha}\cdot (n-1),\omega^{\alpha}\cdot n]\times (\omega^{\alpha}\cdot n,\omega^{\alpha+1}]} \cup \{(\omega^{(\alpha+1)\odot 2}, \omega^{(\alpha+1)\odot 2})\}.
            \end{align*}
            The continuity of the function $U_{\alpha,\beta}(f,g)$ and properties (i) to (v) are checked in a similar way in the rest of the steps, so we will omit the details.
            
            \item {\bf Inductive step ($\beta\to \beta+1$):} Assuming that $U_{\alpha+1,\beta}$ has been established,  let us construct $U_{\alpha+1,\beta+1}$ as follows:
            \begin{align*}
                U_{\alpha+1,\beta+1}(f,g)&|_{(\omega^{\alpha\oplus\beta+1}\cdot(2n-2),\omega^{\alpha\oplus\beta+1}\cdot(2n-1)]} =\, U_{\alpha+1,\beta}\bigl(f|_{(\omega^\alpha\cdot(n-1),\omega^{\alpha+1}]},g|_{(\omega^\beta\cdot(n-1),\omega^{\beta}\cdot n]}\bigr), n\in \N;\\
                U_{\alpha+1,\beta+1}(f,g)&|_{(\omega^{\alpha\oplus\beta+1}\cdot(2n-1),\omega^{\alpha\oplus\beta+1}\cdot 2n]}  =\, U_{\alpha,\beta+1}\bigl(f|_{(\omega^\alpha\cdot(n-1),\omega^{\alpha}\cdot n]},g|_{(\omega^\beta \cdot n,\omega^{\beta+1}]}\bigr) , n\in \N;\\
                U_{\alpha+1,\beta+1}(f,g)&|_{\{\omega^{(\alpha+1)\oplus(\beta+1)}\}}  =\,\frac{1}{2}\bigl(f(\omega^{\alpha+1})+g(\omega^{\beta+1})\bigr).
            \end{align*}
            
            \item {\bf Limit case ($\beta$ limit ordinal):} If $\beta > \alpha+1$ is a limit ordinal and $U_{\alpha+1,\gamma}$ exists for every $\alpha+1\leq\gamma<\beta$, then we can define $U_{\alpha+1,\beta}$ in the following way. First, observe that $\beta$ is countable, so it has cofinality $\omega$ and there is a strictly increasing sequence of ordinals $(\beta_n)_n$ indexed by $n\in \omega$ such that $\beta=\sup_n\beta_n$. Next, note that $\gamma=\sup_n (\alpha+1)\oplus \beta_n$ exists and $\gamma\leq \alpha \oplus \beta$. Indeed, the ordinals are a well-ordered set, so the subset of upper bounds of $\{(\alpha+1)\oplus \beta_n:n\in \omega\}$ must have a minimum, i.e., a least upper bound for the sequence. Moreover, $\beta_n<\beta_{n+1}$ implies that $(\alpha+1)\oplus \beta_n=\alpha\oplus (\beta_n+1)\leq \alpha \oplus \beta_{n+1}\leq \alpha \oplus \beta$ for every $n\in \omega$, so $\gamma\leq \alpha \oplus \beta$. Using this sequence we can make the following decomposition (hereby, $\omega^{\beta_0}$ and $\omega^{(\alpha+1)\oplus \beta_{0}}$ must be understood as the ordinal $0$):
            \begin{align*}
                [1,\omega^{(\alpha+1)\oplus \beta}] & = 
                \intoo[3]{\bigcup_{n=1}^\infty (\omega^{(\alpha+1)\oplus \beta_{n-1}},\omega^{(\alpha+1)\oplus \beta_n}]}
                \cup [\omega^\gamma, \omega^{\alpha\oplus \beta}]\\
                &\cup \intoo[3]{\bigcup_{n=1}^\infty (\omega^{\alpha\oplus \beta}\cdot n,\omega^{\alpha\oplus \beta}\cdot (n+1)]}
            \cup \{\omega^{(\alpha+1)\oplus \beta}\}.
            \end{align*}
            Since $((\alpha+1)\oplus \beta_n)_n$ is strictly increasing, it follows that every clopen set $(\omega^{(\alpha+1)\oplus \beta_{n-1}},\omega^{(\alpha+1)\oplus \beta_n}]$ is homeomorphic to $[1,\omega^{(\alpha+1)\oplus \beta_n}]$. We define
            \begin{align*}
                U_{\alpha+1,\beta}(f,g)&|_{(\omega^{(\alpha+1)\oplus \beta_{n-1}},\omega^{(\alpha+1)\oplus \beta_n}]} =\, U_{\alpha+1,\beta_n}\bigl(f|_{(\omega^\alpha\cdot(n-1),\omega^{\alpha+1}]},g|_{(\omega^{\beta_{n-1}},\omega^{\beta_n}]}\bigr), n\in \N;\\
                U_{\alpha+1,\beta}(f,g)&|_{[\omega^{\gamma},\omega^{\alpha\oplus \beta}]}  =\,\frac{1}{2}\bigl(f(\omega^{\alpha+1})+g(\omega^{\beta})\bigr);\\
                U_{\alpha+1,\beta}(f,g)&|_{(\omega^{\alpha\oplus \beta} n,\omega^{\alpha\oplus \beta} (n+1)]}  =\, U_{\alpha,\beta}\bigl(f|_{(\omega^\alpha\cdot(n-1),\omega^{\alpha}\cdot n]},g|_{(\omega^{\beta_n},\omega^{\beta}]}\bigr) , n\in \N;\\
                U_{\alpha+1,\beta}(f,g)&|_{\{\omega^{(\alpha+1)\oplus\beta}\}}  =\,\frac{1}{2}\bigl(f(\omega^{\alpha+1})+g(\omega^{\beta})\bigr).
            \end{align*}
        \end{enumerate}
        We have now defined $U_{\alpha+1,\beta}$ for every $\beta\geq \alpha+1$, so the proof of the inductive step (IH-I) is concluded.
        
        \item {\bf Limit case ($\alpha$ limit ordinal):} Assume that $\alpha$ is a limit ordinal, and (IH) has been established for every $\gamma<\alpha$. In particular, since $\alpha$ has cofinality $\omega$, there is a strictly increasing sequence of ordinals $(\alpha_n)_{n\in \omega}$ such that $\alpha=\sup_n\alpha_n$. We construct $U_{\alpha,\beta}$ by transfinite induction on $\beta\geq \alpha$:\medskip

        \noindent {\bf Inductive Hypothesis (IH-L):} Fix $\alpha \in \omega_1$ a limit ordinal, and assume that $U_{\gamma,\delta}$ has been constructed for every $\gamma<\alpha$ and $\delta\geq \gamma$. Let $\beta\geq \alpha$. If $U_{\alpha,\gamma}$ is defined for every $\alpha\leq \gamma <\beta$, then $U_{\alpha,\beta}$ exists.\medskip

        \begin{enumerate}[3.1)]
            \item {\bf Base case ($\beta=\alpha$):} We can construct $U_{\alpha,\alpha}$ in the following way (as before, $\omega^{\alpha_0}$ and $\omega^{\alpha_{0}\oplus\alpha}$ denote the $0$ ordinal):
            \begin{align*}
                U_{\alpha,\alpha}(f,g)&|_{(\omega^{\alpha_{n-1}\oplus\alpha}\cdot2,\omega^{\alpha_n \oplus\alpha}]} =\, U_{\alpha,\alpha_n}\bigl(f|_{(\omega^{\alpha_{n-1}},\omega^{\alpha}]},g|_{(\omega^{\alpha_{n-1}},\omega^{\alpha_n}]}\bigr), n\in \N;\\
                U_{\alpha,\alpha}(f,g)&|_{(\omega^{\alpha_{n} \oplus\alpha},\omega^{\alpha_n \oplus\alpha}\cdot2]} =\, U_{\alpha_n,\alpha}\bigl(f|_{(\omega^{\alpha_{n-1}},\omega^{\alpha_n}]},g|_{(\omega^{\alpha_n},\omega^{\alpha}]}\bigr) , n\in \N;\\
                U_{\alpha,\alpha}(f,g)&|_{\{\omega^{\alpha\odot 2}\}} =\,\frac{1}{2}\bigl(f(\omega^{\alpha})+g(\omega^{\alpha})\bigr).
            \end{align*}
            
            \item {\bf Inductive step ($\beta\to \beta+1$):} Assume that $U_{\alpha,\beta}$ is known. We need to construct $U_{\alpha,\beta+1}$. First observe that the sequence $\alpha_n\oplus (\beta+1)$ is bounded below by $\alpha\oplus \beta$. Let $\gamma=\sup_n \alpha_n\oplus (\beta+1)\leq \alpha \oplus \beta$. Then, we can make the following decomposition (we are using the convention $\omega^{\alpha_{0}}=\omega^{\alpha_{0}\oplus(\beta+1)}=0$):
            \begin{align*}
                [1,\omega^{\alpha\oplus (\beta+1)}] & = 
                \intoo[3]{\bigcup_{n=1}^\infty (\omega^{\alpha_{n-1}\oplus (\beta+1)},\omega^{\alpha_n\oplus (\beta+1)}]}
                \cup [\omega^\gamma, \omega^{\alpha\oplus \beta}]\\
                &\cup \intoo[3]{\bigcup_{n=1}^\infty (\omega^{\alpha\oplus \beta}\cdot n,\omega^{\alpha\oplus \beta}\cdot (n+1)]}
            \cup \{\omega^{\alpha\oplus (\beta+1)}\}.
            \end{align*}
            Now, we define $U_{\alpha,\beta+1}$ piecewise as:
            \begin{align*}
                U_{\alpha,\beta+1}(f,g)&|_{(\omega^{\alpha_{n-1}\oplus (\beta+1)},\omega^{\alpha_n\oplus (\beta+1)}]} =\, U_{\alpha_n,\beta+1}\bigl(f|_{(\omega^{\alpha_{n-1}},\omega^{\alpha_n}]},g|_{(\omega^{\beta}\cdot n, \omega^{\beta+1}]}\bigr), n\in \N;\\
                U_{\alpha,\beta+1}(f,g)&|_{[\omega^{\gamma},\omega^{\alpha\oplus \beta}]}  =\,\frac{1}{2}\bigl(f(\omega^{\alpha})+g(\omega^{\beta+1})\bigr);\\
                U_{\alpha,\beta+1}(f,g)&|_{(\omega^{\alpha\oplus \beta} \cdot n,\omega^{\alpha\oplus \beta} \cdot(n+1)]} =\, U_{\alpha,\beta}\bigl(f|_{(\omega^{\alpha_{n-1}},\omega^{\alpha}]},g|_{(\omega^{\beta}\cdot (n-1),\omega^{\beta}\cdot n]}\bigr) , n\in \N;\\
                U_{\alpha,\beta+1}(f,g)&|_{\{\omega^{\alpha\oplus(\beta+1)}\}}  =\,\frac{1}{2}\bigl(f(\omega^{\alpha})+g(\omega^{\beta+1})\bigr).
            \end{align*}
            \item {\bf Limit case ($\beta$ limit ordinal):} Assume that $\beta > \alpha$ is a limit ordinal and let $(\beta_n)_n$ be a strictly increasing sequence indexed by $\omega$ whose supremum is $\beta$. Provided $U_{\alpha,\gamma}$ exists for every $\alpha\leq\gamma<\beta$, we want to define $U_{\alpha,\beta}$. Let $\gamma_n=\max\{\alpha_n\oplus \beta,\alpha\oplus\beta_n\}$, which defines a strictly increasing sequence whose supremum, denoted by $\gamma$, is bounded by $\alpha\oplus\beta$. We make the following decomposition of $[1,\omega^{\alpha\oplus \beta}]$ (here, $\omega^{\gamma_0}=0$):
            \[[1,\omega^{\alpha\oplus \beta}]  = \intoo[3]{\bigcup_{n=1}^\infty (\omega^{\gamma_{n-1}}\cdot 2,\omega^{\gamma_n}]\cup (\omega^{\gamma_{n}},\omega^{\gamma_n}\cdot 2]}  \cup [\omega^\gamma,\omega^{\alpha\oplus \beta}].\]
            Additionally, we can break each of the clopen sets above as follows:
            \begin{align*}
                (\omega^{\gamma_{n-1}}\cdot 2,\omega^{\gamma_n}]&=(\omega^{\gamma_{n-1}}\cdot 2,\omega^{\gamma_{n-1}}\cdot 2+\omega^{\alpha\oplus\beta_n}]\cup (\omega^{\gamma_{n-1}}\cdot 2+\omega^{\alpha\oplus\beta_n},\omega^{\gamma_n}], n\in \N;\\
                (\omega^{\gamma_n},\omega^{\gamma_n}\cdot 2]&=(\omega^{\gamma_n},\omega^{\gamma_n}+\omega^{\alpha_n\oplus\beta}]\cup (\omega^{\gamma_n}+\omega^{\alpha_n\oplus\beta},\omega^{\gamma_n}\cdot 2], n\in \N;
            \end{align*}
            where the last set in each line is either empty or a clopen set. Now, we define $U_{\alpha,\beta}$ piecewise as follows:
            \begin{align*}
                U_{\alpha,\beta}(f,g)&|_{(\omega^{\gamma_{n-1}}\cdot 2,\omega^{\gamma_{n-1}}\cdot 2+\omega^{\alpha\oplus\beta_n}]} =\, U_{\alpha,\beta_n}\bigl(f|_{(\omega^{\alpha_{n-1}},\omega^{\alpha}]},g|_{(\omega^{\beta_{n-1}}, \omega^{\beta_n}]}\bigr), n\in \N;\\
                U_{\alpha,\beta}(f,g)&|_{(\omega^{\gamma_{n-1}}\cdot 2+\omega^{\alpha\oplus\beta_n},\omega^{\gamma_n}]}  =\,\frac{1}{2}\bigl(f(\omega^{\alpha})+g(\omega^{\beta_n})\bigr), n\in \N;\\
                U_{\alpha,\beta}(f,g)&|_{(\omega^{\gamma_n},\omega^{\gamma_n}+\omega^{\alpha_n\oplus\beta}]}  =\, U_{\alpha_n,\beta}\bigl(f|_{(\omega^{\alpha_{n-1}},\omega^{\alpha_n}]},g|_{(\omega^{\beta_n},\omega^{\beta}]}\bigr) , n\in \N;\\
                U_{\alpha,\beta}(f,g)&|_{(\omega^{\gamma_n}+\omega^{\alpha_n\oplus\beta},\omega^{\gamma_n}\cdot2]}  =\,\frac{1}{2}\bigl(f(\omega^{\alpha_n})+g(\omega^{\beta})\bigr), n\in \N;\\
                U_{\alpha,\beta}(f,g)&|_{[\omega^\gamma,\omega^{\alpha\oplus \beta}]}  =\,\frac{1}{2}\bigl(f(\omega^{\alpha})+g(\omega^{\beta})\bigr).
            \end{align*}
        \end{enumerate}
        With this last construction we have defined $U_{\alpha,\beta}$ for every $\beta\geq \alpha$, so (IH-L) holds, and hence so does (IH) for every countable ordinal $\alpha$. \qedhere
    \end{enumerate}
\end{proof}

Now that we have built an overlapping map for every ordinal $\alpha$, we can use it to obtain the following isometric and isomorphic results in the real setting.

\begin{theorem}\label{thm: real SPR copy of omega^alpha in omega^alpha2}
    Let $K$ be a compact Hausdorff space, $\alpha\in \omega_1$ a countable ordinal and assume that the field of scalars is $\R$. If $K^{(\alpha\odot 2)}$ is non-empty, then there exists an isometric $3$-SPR embedding of $C[1,\omega^\alpha]$ into $C(K)$.
\end{theorem}

\begin{proof}
    By \Cref{prop: implication K_alpha copy of Comegaalpha}, $C(K)$ contains an isometric copy of $C[1,\omega^{\alpha\odot 2}]$ that preserves SPR subspaces, so without loss of generality we can assume that $K=[1,\omega^{\alpha\odot 2}]$. By \Cref{lemma: construction of map U}, there is a map $U_{\alpha,\alpha}:C[1,\omega^{\alpha}]\times C[1,\omega^{\alpha}]\rightarrow C[1,\omega^{\alpha\odot 2}]$ satisfying properties (i) to (v). Let
    \[\fullfunction{T}{C[1,\omega^{\alpha}]}{C[1,\omega^{\alpha\odot 2}]}{f}{Tf=U_{\alpha,\alpha}(f,f).}\]
    Properties (i) and (ii) imply that $T$ is a linear map. Moreover, (iii) implies that $T$ is contractive, and (iv) yields that for every $s\in [1,\omega^{\alpha}]$ there exists $r_{s,s}\in [1,\omega^{\alpha\odot 2}]$ such that $Tf(r_{s,s})=f(s)$, so $T$ is an isometric embedding. It remains to show that $T(C[1,\omega^{\alpha}])$ does $3$-SPR in $C[1,\omega^{\alpha\odot 2}]$. Let $f$ and $g$ be normalized functions in $C[1,\omega^{\alpha}]$, and find $s,t\in [1,\omega^{\alpha}]$ such that $|f(s)|=|g(t)|=1$.
    \begin{itemize}
        \item If $s=t$, then $|Tf(r_{s,s})|\wedge|Tg(r_{s,s})|=|f(s)|\wedge|g(s)|=1$.
        \item If $s\neq t$, we distinguish three cases:
        \begin{itemize}
            \item If $|Tf(r_{s,t})|<\frac{1}{3}$, then 
            \[|f(t)|\geq |f(s)|-|f(s)+f(t)|=1-2|Tf(r_{s,t})|>\frac{1}{3},\]
            so $|Tf(r_{t,t})|\wedge|Tg(r_{t,t})|=|f(t)|\wedge|g(t)|\geq \frac{1}{3}$.
            \item If $|Tg(r_{s,t})|<\frac{1}{3}$, then 
            \[|g(s)|\geq |g(t)|-|g(s)+g(t)|=1-2|Tg(r_{s,t})|>\frac{1}{3},\]
            so $|Tf(r_{s,s})|\wedge|Tg(r_{s,s})|=|f(s)|\wedge|g(s)|\geq \frac{1}{3}$.
            \item Otherwise, $|Tf(r_{s,t})|\wedge|Tg(r_{s,t})|\geq \frac{1}{3}$.
        \end{itemize}
    \end{itemize}
    Hence, there is always a point $r\in [1,\omega^{\alpha\odot 2}]$ such that $\||Tf|\wedge|Tg|\|_\infty \geq |Tf(r)|\wedge|Tg(r)|\geq \frac{1}{3}$, and we conclude that $T(C[1,\omega^{\alpha}])$ does not contain normalized $\frac{1}{3}$-almost disjoint pairs, so, by \cite[Proposition 3.4]{Eugene}, it does $3$-SPR. 
\end{proof}

\begin{corr}\label{coro: real best SPR embedding into omega^alpha}
    Let $0<\alpha< \omega_1$ be a countable ordinal, $K$ a compact Hausdorff space with $K^{(\alpha)}\neq \varnothing$, and assume the field of scalars is $\R$. Let $\beta\in \omega_1$ such that $\omega^\beta\leq \alpha <\omega^{\beta+1}$. Then, $C(K)$ contains an isometric SPR copy of $C[1,\omega^\gamma]$ for every $\gamma<\omega^\beta$. Moreover, if $\omega^\beta\cdot 2\leq \alpha$, then $C(K)$ contains an isomorphic SPR copy of $C[1,\omega^\alpha]$.
\end{corr}

\begin{proof}
    Using Cantor's normal form \eqref{eq: cantor normal form} it can be easily checked that for every $\gamma<\omega^\beta$, we have $\gamma\odot 2<\omega^\beta\leq \alpha$, so using \Cref{thm: real SPR copy of omega^alpha in omega^alpha2} we obtain a $3$-SPR isometric embedding of $C[1,\omega^\gamma]$ into $C[1,\omega^{\gamma\odot 2}]$, which in turn is a sublattice of $C[1,\omega^\alpha]$. By \Cref{prop: implication K_alpha copy of Comegaalpha}, $C(K)$ contains a $3$-SPR isometric copy of $C[1,\omega^\gamma]$ as well.\medskip
    
    Next, assume that $\omega^\beta\cdot 2\leq \alpha$. It was established in \cite[Theorem 1]{BP} that for ordinals $\omega \leq \lambda\leq \mu <\omega_1$, $C[1,\lambda]$ is isomorphic (as a Banach space) to $C[1,\mu]$ if and only if $\mu < \lambda^\omega$. In particular, it can be checked that $\omega^\alpha <(\omega^{\omega^\beta})^\omega$, so $C[1,\omega^\alpha]$ and $C[1,\omega^{\omega^\beta}]$ are linearly isomorphic. By \Cref{thm: real SPR copy of omega^alpha in omega^alpha2}, we know that $C[1,\omega^{\omega^\beta\cdot 2}]$, which is a sublattice of $C[1,\omega^\alpha]$, contains an SPR copy of $C[1,\omega^{\omega^\beta}]$, and hence, of $C[1,\omega^\alpha]$. \Cref{prop: implication K_alpha copy of Comegaalpha} yields the result.
\end{proof}

In the complex case, the situation is not as straightforward, but it still allows for a quantitative answer to \Cref{question:timur's-question}.

\begin{theorem}\label{thm: complex SPR copy of omega^alpha in 2omega^alpha2}
    Let $K$ be a compact Hausdorff space, $\alpha\in \omega_1$ a countable ordinal and assume that the field of scalars is $\C$. If $|K^{(\alpha\odot 2)}|\geq 2$, then there exists an isometric SPR embedding of $C[1,\omega^\alpha]$ into $C(K)$, and the SPR constant is uniform in $\alpha$.
\end{theorem}

\begin{proof}
    First, observe that from \Cref{prop: implication K_alpha copy of Comegaalpha} it follows that if $|K^{(\alpha\odot 2)}|\geq 2$, there is a positive linear isometric embedding $T:C[1,\omega^\beta\cdot 2]\to C(K)$ satisfying the property \eqref{property *}.
    Thus, we can assume without loss of generality that $K=[1,\omega^{\alpha \odot 2}\cdot 2]$. 
    
    Next, we use \Cref{lemma: construction of map U} to build the operator
    \[\fullfunction{T}{C[1,\omega^{\alpha}]}{C[1,\omega^{\alpha\odot 2}\cdot 2]=C[1,\omega^{\alpha\odot 2}]\oplus_\infty C[1,\omega^{\alpha\odot 2}]}{f}{Tf=(U_{\alpha,\alpha}(f,f),U_{\alpha,\alpha}(f,if)).}\]
    It is clear that $T$ is a contractive map that satisfies the conditions of \Cref{lemma: suficient condition for complex SPR}, so it is an SPR embedding. Moreover, the SPR constant is given by the proofs of \Cref{lemma: suficient condition for complex SPR} and \Cref{thm: complex stable phase retrieval}, so it does not depend on $\alpha$.
\end{proof}

Repeating the argument of \Cref{coro: real best SPR embedding into omega^alpha} we obtain the following:

\begin{corr}\label{coro: complex best SPR embedding into omega^alpha}
    Let $0<\alpha< \omega_1$ be a countable ordinal, $K$ a compact Hausdorff space with $K^{(\alpha)}\neq \varnothing$, and assume that the field of scalars is $\C$. Let $\beta\in \omega_1$ be such that $\omega^\beta\leq \alpha <\omega^{\beta+1}$. Then $C(K)$ contains an isometric SPR  copy of $C[1,\omega^\gamma]$ for every $\gamma<\omega^\beta$. Moreover, if $|K^{(\omega^\beta\cdot 2)}|\geq 2$, that is, if $\alpha > \omega^\beta\cdot 2$ or $\alpha = \omega^\beta\cdot 2$ and $|K^{(\alpha)}|\geq 2$, then $C(K)$ contains an isomorphic SPR  copy of $C[1,\omega^\alpha]$.
\end{corr}

\section*{Acknowledgements}
The research of E.~Garc\'ia-S\'anchez and D.~de~Hevia is partially supported by the grants PID2020-116398GB-I00, PID2024-162214NB-I00, CEX2019-000904-S and CEX2023-001347-S funded by the MICIU/AEI/10.13039/501100011033. 
E.~Garc\'ia-S\'anchez is partially supported by the grant CEX2019-000904-S-21-3 funded by the MICIU/AEI/ 10.13039/501100011033 and by ``ESF+''. D.~de Hevia benefited from an FPU Grant FPU20/03334 from the Ministerio de Ciencia, Innovación y Universidades.\medskip

The authors would like to thank Antonio Avilés, Eugene Bilokopytov, Jesús Illescas, Timur Oikhberg, Alberto Salguero, and Pedro Tradacete for their valuable discussions and suggestions.

\end{document}